\documentclass[11pt,a4paper]{article}

\usepackage{epsf,epsfig,amsfonts,amsgen,amsmath,amstext,amsbsy,amsopn,amsthm,cases,listings,color
}
\usepackage{ebezier,eepic}
\usepackage{color}
\usepackage{multirow}
\usepackage{epstopdf}
\usepackage{graphicx}   
\usepackage{pgf,tikz}
\usepackage{mathrsfs}
\usepackage[marginal]{footmisc}
\usepackage{enumitem}
\usepackage[titletoc]{appendix}
\usepackage{booktabs}
\usepackage{url}
\usepackage{mathtools}

\usepackage{pgfplots}
\usepackage{authblk}
\usepackage{amssymb}

\usepackage{wasysym}

\usepackage{empheq}

\usepackage{dsfont}

\usepackage{tikz}
\usepackage{longtable}

\usepackage{subfigure}
\pgfplotsset{compat=1.18}
\usepackage{mathrsfs}
\usepackage{wasysym} 
\usetikzlibrary{arrows}
\usepackage{aligned-overset}
\usepackage{bm}
\usepackage{bbm}
\usepackage[T1]{fontenc}
\usepackage[backref=page]{hyperref}

\allowdisplaybreaks[1]

\definecolor{uuuuuu}{rgb}{0.27,0.27,0.27}
\definecolor{sqsqsq}{rgb}{0.1255,0.1255,0.1255}

\newtheorem{definition}{Definition} [section]
\newtheorem{theorem}[definition]{Theorem}
\newtheorem{lemma}[definition]{Lemma}
\newtheorem{proposition}[definition]{Proposition}

\newtheorem{corollary}[definition]{Corollary}

\newtheorem{claim}[definition]{Claim}

\newtheorem{fact}[definition]{Fact}

\newenvironment{pf}{\noindent{\bf Proof}}

\newcommand{\norm}[1]{\left\lVert#1\right\rVert}

\tikzset{unlabeled_vertex/.style={inner sep=1.7pt, outer sep=0pt, circle, fill}} 
\tikzset{labeled_vertex/.style={inner sep=2.2pt, outer sep=0pt, rectangle, fill=yellow, draw=black}} 
\tikzset{edge_color0/.style={color=black,line width=1.2pt,opacity=0.5}} 
\tikzset{edge_color1/.style={color=red,  line width=1.2pt,opacity=1}} 
\tikzset{edge_color2/.style={color=blue, line width=1.2pt,opacity=1}} 
\tikzset{edge_color3/.style={color=green,line width=1.2pt}} 
\tikzset{edge_color4/.style={color=red,  line width=1.2pt,dotted}} 
\tikzset{edge_color5/.style={color=blue, line width=1.2pt,dotted}} 
\tikzset{edge_color6/.style={color=green, line width=1.2pt,dotted}} 
\tikzset{edge_color7/.style={color=orange, line width=1.2pt}} 
\tikzset{edge_color8/.style={color=gray, line width=1.2pt}} 
\tikzset{edge_thin/.style={color=black}} 
\tikzset{edge_hidden/.style={color=black,dotted,opacity=0}} 
\tikzset{vertex_color1/.style={inner sep=1.7pt, outer sep=0pt, draw, circle, fill=red}} 
\tikzset{vertex_color2/.style={inner sep=1.7pt, outer sep=0pt, draw, circle, fill=blue}} 
\tikzset{vertex_color3/.style={inner sep=1.7pt, outer sep=0pt, draw, circle, fill=green}} 
\tikzset{labeled_vertex_color1/.style={inner sep=2.2pt, outer sep=0pt, draw, rectangle, fill=red}} 
\tikzset{labeled_vertex_color2/.style={inner sep=2.2pt, outer sep=0pt, draw, rectangle, fill=blue}} 
\tikzset{labeled_vertex_color3/.style={inner sep=2.2pt, outer sep=0pt, draw, rectangle, fill=green}}

\tikzset{
vtx/.style={inner sep=1.1pt, outer sep=0pt, circle, fill,draw}, 
vtxl/.style={inner sep=1.1pt, outer sep=0pt, rectangle, fill=yellow,draw=black}, 
hyperedge/.style={fill=pink,opacity=0.5,draw=black}, 
}

\begin{document}
\title{\bf\Large Strong stability from vertex-extendability and applications in generalized Tur\'{a}n problems}
\date{\today}
\author[1,2]{Wanfang Chen\thanks{Research was supported by China Scholarship Council (CSC)~No.~202306140125. \\ Email: \texttt{52215500039@stu.ecnu.edu.cn}}}
\author[1]{Xizhi Liu\thanks{Research was supported by ERC Advanced Grant 101020255. Email: \texttt{xizhi.liu.ac@gmail.com}}}
\affil[1]{Mathematics Institute and DIMAP,
            University of Warwick, 
            Coventry, CV4 7AL, UK}
\affil[2]{School of Mathematical Sciences and Shanghai Key Laboratory of PMMP, 
East China Normal University, 
Shanghai, 200241, China}
\maketitle
\begin{abstract}
Extending the work of Liu--Mubayi--Reiher~\cite{LMR23unif} on hypergraph Tur\'{a}n problems, we introduce the notion of vertex-extendability for general extremal problems on hypergraphs and develop an axiomatized framework for proving strong stability for extremal problems satisfying certain properties. 
This framework simplifies the typically complex and tedious process of obtaining stability and exact results for extremal problems into a much simpler task of verifying their vertex-extendability.

We present several applications of this method in generalized Tur\'{a}n problems including the Erd\H{o}s Pentagon Problem, hypergraph Tur\'{a}n-goodness, and generalized Tur\'{a}n problems of hypergraphs whose shadow is complete multipartite. 
These results significantly strengthen and extend previous results of Erd\H{o}s~\cite{Erdos62}, Gy\H{o}ri--J\'{a}nos--Simonovits~\cite{GPS91}, Grzesik~\cite{Gre12}, Hatami--Hladk\'{y}--Kr\'{a}\v{l}--Norine--Razborov~\cite{HHKNR13},  Morrison--Nir--Norin--Rz\k{a}\.{z}ewski--Wesolek~\cite{MNNRPW23}, Gerbner--Palmer~\cite{GP22}, and others.  

\medskip



\end{abstract}
\section{Intorduction}\label{SEC:Intorduction}
For an integer $r\ge 2$, an \textbf{$r$-uniform hypergraph} (henceforth \textbf{$r$-graph}) $\mathcal{H}$ is a collection of $r$-subsets of some finite set $V$.
We identify a hypergraph $\mathcal{H}$ with its edge set and use $V(\mathcal{H})$ to denote its vertex set. 
The size of $V(\mathcal{H})$ is denoted by $v(\mathcal{H})$. 
For every $i \in [r-1]$, the \textbf{link} of an $i$-set $T \subset V(\mathcal{H})$ is defined as 
\begin{align*}
    L_{\mathcal{H}}(T)
    \coloneqq \left\{e \in \binom{V(\mathcal{H})}{r-i} \colon T \cup e \in \mathcal{H}\right\}. 
\end{align*}
The \textbf{degree} of $T$ in $\mathcal{H}$ is $d_{\mathcal{H}}(T) \coloneqq |L_{\mathcal{H}}(T)|$. 
We use $\delta_{i}(\mathcal{H})$, $\Delta_{i}(\mathcal{H})$, and $d_{i}(\mathcal{H})$ to denote the \textbf{minimum, maximum, and average $i$-degree} of $\mathcal{H}$, respectively.
We will omit the subscript $i$ in the case where $i = 1$. 

We use $\mathfrak{G}^{r}$ to denote the collection of all vertex-labeled\footnote{We assume that different vertices receive different labels.} $r$-graphs, and for every $n \in \mathbb{N}$, we use $\mathfrak{G}^{r}(n) \subset \mathfrak{G}^{r}$ to denote the subfamily consisting of all $n$-vertex $r$-graphs. 
Given a map $\Gamma \colon \mathfrak{G}^{r} \to \mathbb{R}$, 
the \textbf{$\Gamma$-degree} of a vertex $v$ in an $r$-graph $\mathcal{H}$ is defined as  
\begin{align*}
    d_{\Gamma, \mathcal{H}}(v)
    \coloneqq \Gamma(\mathcal{H}) - \Gamma(\mathcal{H}-v),  
\end{align*}
where $\mathcal{H}-v$ denotes the $r$-graph obtained from $\mathcal{H}$ by removing $v$ and all edges containing $v$. 
We use $\delta_{\Gamma}(\mathcal{H})$, $\Delta_{\Gamma}(\mathcal{H})$, and $d_{\Gamma}(\mathcal{H})$ to denote the \textbf{minimum, maximum, and average $\Gamma$-degree} of $\mathcal{H}$, respectively. 
We say an $r$-graph $\mathcal{H}$ is \textbf{$\Gamma$-positive} if $\Gamma(\mathcal{H})>0$. 

In Extremal Combinatorics, given a map $\Gamma \colon \mathfrak{G}^{r} \to \mathbb{R}$, we are interested in the \textbf{extremal number} $\mathrm{ex}_{\Gamma}(n)$ and the \textbf{extremal family} $\mathrm{EX}_{\Gamma}(n)$ of $\Gamma$, which are defined by 
\begin{align*}
    \mathrm{ex}_{\Gamma}(n) 
     \coloneqq \max\left\{\Gamma(\mathcal{H}) \colon \mathcal{H} \in \mathfrak{G}^{r}(n)\right\},
     \quad  
    \mathrm{EX}_{\Gamma}(n)
     \coloneqq \left\{\mathcal{H} \in \mathfrak{G}^{r}(n) \colon \Gamma(\mathcal{H}) = \mathrm{ex}_{\Gamma}(n)\right\}.
\end{align*}
Several central topics in Extremal Combinatorics, such as the Tur\'{a}n problem, generalized Tur\'{a}n problem, and the inducibility problem, can be formulated as determining $\mathrm{ex}_{\Gamma}(n)$ and $\mathrm{EX}_{\Gamma}(n)$ for specific maps $\Gamma$ (see Section~\ref{SEC:applications} for some examples).
Prior research has established foundational methodologies for deriving exact and stability theorems for various specific extremal problems. 
Notably, references~\cite{NY17,PST19,LMR23unif} address nondegenerate hypergraph Tur\'{a}n problems, and reference~\cite{LPSS23} handles graph extremal problems with applications in the graph inducibility problem. 

In this work, we build upon the foundation established by Liu--Mubayi--Reiher~\cite{LMR23unif} regarding nondegenerate hypergraph Tur\'{a}n problems, extending it to a much broader class of extremal problems concerning hypergraphs. 
Our primary contributions introduce an axiomatized framework for proving a strong form of stability known as degree-stability for general hypergraph extremal problems.
This form of stability was initially explored by Andr\'{a}sfai--Erd\H{o}s--S\'{o}s~\cite{AES74} and Erd\H{o}s--Simonovits~\cite{ES73} in the context of graph Tur\'{a}n problems.
It is stronger than the type of stability first considered by Simonovits in his pioneering work~\cite{SI68}, and it often leads to an exact result almost effortlessly. 
A key concept in this framework is what we call $\Gamma$-vertex-extendability, which is an extension of the vertex-extendability introduced in~\cite{LMR23unif} for nondegenerate Tur\'{a}n problems. 
Our main results (Theorems~\ref{THM:GeneralFunction-Stability},~\ref{THM:GeneralFunction-Stability-b}, and~\ref{THM:GeneralFunction-Stability-c}) in this section reduce the proof for the degree-stability of extremal problem associated with $\Gamma$ to the straightforward task of verifying the vertex-extendability of $\Gamma$. 

In the subsequent part of this section, we will introduce definitions inspired by common properties shared by some well-studied extremal problems and state the main results. 
In Section~\ref{SEC:applications}, we will present some applications of the general theorems in generalized Tur\'{a}n problems.  
Additionally, in an upcoming paper~\cite{CILLP24}, we will show applications of our general theorems to Tur\'{a}n problems under $(t,p)$-norms (see Section~\ref{SEC:Remark} for more details).  
It is worth noting that there are several additional extremal problems, such as those considered in~\cite{LM21,BJ23,GGJ23,Ger23}, which our general theorems could potentially address. 
These applications are left for interested readers to explore in the future.

Throughout this paper, asymptotic notations are taken with respect to $n$.
Though not entirely precise, it may be helpful for readers to understand the definitions and calculations in the proofs by considering  $\varepsilon, \delta$ as sufficiently small constants and $n$ as a sufficiently large integer. 

\subsection{Definitions}
We use $\mathbb{R}_{\ge 0}$ to denote the set of nonnegative real numbers. 
Given a real number $k \in \mathbb{R}$, we say 
a map $\Gamma \colon \mathfrak{G}^{r} \to \mathbb{R}$ is \textbf{$k$-uniform} if for every $n \in \mathbb{N}$ and for every $\Gamma$-positive $\mathcal{H} \in \mathfrak{G}^{r}(n)$, 
\begin{align}\label{equ:k-uniform-def}
    \left| \sum_{v\in V(\mathcal{H})}d_{\Gamma, \mathcal{H}}(v) - k \cdot \Gamma(\mathcal{H}) \right|
    = o\left(\mathrm{ex}_{\Gamma}(n)\right).  
\end{align}
The \textbf{extremal average degree} of a $k$-uniform map $\Gamma \colon \mathfrak{G}^{r} \to \mathbb{R}$ is defined as
\begin{align*}
    \mathrm{exdeg}_{\Gamma}(n)
    \coloneqq \frac{k \cdot \mathrm{ex}_{\Gamma}(n)}{n}. 
\end{align*}
\begin{definition}\label{DEF:general-property-a}
    Let $\Gamma \colon \mathfrak{G}^{r} \to \mathbb{R}_{\ge 0}$ be a $k$-uniform map.  
    \begin{enumerate}[label=(\roman*)]
        \item\label{DEF:general-property-a1} We say $\Gamma$ is \textbf{$(\delta_{\ast}, c_{\ast},C_{\ast})$-smooth} for some $\delta_{\ast} \in (0,1]$, $c_{\ast}>0$, and $C_{\ast} > 0$ if, for every $n \in \mathbb{N}$ and $\delta \in [0,\delta_{\ast}]$, 
            \begin{align*}
                \mathrm{ex}_{\Gamma}(n-\delta n) 
                \le \mathrm{ex}_{\Gamma}(n) - \delta k \cdot \mathrm{ex}_{\Gamma}(n) + C_{\ast} \delta^{1+c_{\ast}} \cdot \mathrm{ex}_{\Gamma}(n) + o\left(\mathrm{ex}_{\Gamma}(n)\right). 
            \end{align*}
        \item\label{DEF:general-property-a2} We say $\Gamma$ is \textbf{locally $C$-Lipschitz} for some $C >0$ if, 
        for every $\Gamma$-positive $n$-vertex $r$-graph $\mathcal{H}$ and for every induced subgraph $\mathcal{H}'\subset \mathcal{H}$ on $n' \le n$ vertices, the following inequality holds for all $v\in V(\mathcal{H}') \colon$ 
        \begin{align}\label{equ:DEF:general-property-a2}
            d_{\Gamma, \mathcal{H}'}(v)
            \ge 
            d_{\Gamma, \mathcal{H}}(v) - \frac{C(n-n')}{n} \cdot \mathrm{exdeg}_{\Gamma}(n)
            - o(\mathrm{degex}_{\Gamma}(n)). 
        \end{align}
    \end{enumerate}
\end{definition}
We will omit parameters $k, \delta_{\ast}, c_{\ast},C_{\ast}, C$ if they are not crucial in the context. 

\begin{definition}\label{DEF:general-property-b}
    Let $\Gamma \colon \mathfrak{G}^{r} \to \mathbb{R}$ be a map.   
\begin{enumerate}[label=(\roman*)]
    \item\label{DEF:general-property-b1} We say $\Gamma$ is \textbf{continuous} if for every $\varepsilon>0$ there exist $\delta>0$ and $N_{0}$ such that the following holds for all $r$-graph $\mathcal{H}$ on $n\ge N_{0}$ vertices$\colon$
    if $\mathcal{H}' \subset \mathcal{H}$ satisfies $|\mathcal{H}'| \ge |\mathcal{H}| - \delta n^r$, then $\Gamma(\mathcal{H}') \ge \Gamma(\mathcal{H}) - \varepsilon \cdot \mathrm{ex}_{\Gamma}(n)$,
    \item\label{DEF:general-property-b2} We say $\Gamma$ is \textbf{locally monotone} if for every $\Gamma$-positive $r$-graph $\mathcal{H}$ and for every subgrpah $\mathcal{H}' \subset \mathcal{H}$, the inequality $d_{\Gamma, \mathcal{H}'}(v) \le d_{\Gamma, \mathcal{H}}(v)$ holds for every $v\in V(\mathcal{H}')$.
\end{enumerate}
\end{definition}
The following definitions are motivated by the well-known Zykov symmetrization~\cite{Zyk49}. 
Given a pair of vertices $u,v\in V(\mathcal{H})$,
we say $\{u,v\}$ is \textbf{uncovered} if there is no edge in $\mathcal{H}$ containing both $u$ and $v$, and say $u,v$ are \textbf{equivalent} if $L_{\mathcal{H}}(u) = L_{\mathcal{H}}(v)$ (in particular, this implies that $\{u,v\}$ is uncovered).
We use $\mathcal{H}_{u\to v}$ to denote the $r$-graph obtained from $\mathcal{H}$ by \textbf{symmetrizing $u$ into $v$}, that is,  
\begin{align*}
    \mathcal{H}_{u\to v} 
    \coloneqq \left(\mathcal{H} \setminus \{E\in \mathcal{H} \colon u\in E\} \right) \cup \left\{\{u\}\cup e \colon e\in L_{\mathcal{H}}(v) \right\}. 
\end{align*}
An $r$-graph $\mathcal{H}$ is \textbf{symmetrized} if every pair of uncovered vertices in $\mathcal{H}$ are equivalent.

\begin{definition}\label{DEF:symmetrization-increasing}
    A map $\Gamma \colon \mathfrak{G}^{r} \to \mathbb{R}$ is \textbf{symmetrization-increasing} if for every $\mathcal{H} \in \mathfrak{G}^{r}$ and for every pair of uncovered vertices $u,v \in V(\mathcal{H})$,  
\begin{align}\label{equ:sym-increase}
    \text{either}\quad 
    \Gamma(\mathcal{H}) 
    < \max\left\{\Gamma(\mathcal{H}_{u\to v}),\ \Gamma(\mathcal{H}_{v\to u})\right\}
    \quad\text{or}\quad 
    \Gamma(\mathcal{H})
    = \Gamma(\mathcal{H}_{u\to v})
    = \Gamma(\mathcal{H}_{v\to u}).
\end{align}   
\end{definition}
Notice that~\eqref{equ:sym-increase} holds iff there exists $\alpha \in (0,1)$ (in the proofs we will choose $\alpha = 1/2$) such that 
\begin{align*}
    \Gamma(\mathcal{H})
    \le \alpha \cdot \Gamma(\mathcal{H}_{u\to v}) + (1-\alpha) \cdot \Gamma(\mathcal{H}_{v\to u}). 
\end{align*}

\begin{definition}
    Let $\mathfrak{H}$ be a hereditary\footnote{Hereditary means that if $F \in \mathfrak{H}$, then $F' \in \mathfrak{H}$ for all (not necessarily induced) subgraphs $F' \subset F$.} family of $r$-graphs. 
     Let $\Gamma \colon \mathfrak{G}^{r} \to \mathbb{R}$ be a $k$-uniform map.  
     \begin{enumerate}[label=(\roman*)]
        \item We say $\Gamma$ is \textbf{edge-stable} with respect to $\mathfrak{H}$ if for every $\varepsilon>0$ there exist $\delta>0$ and $N_{0}$ such that every $r$-graph on $n \ge N_{0}$ vertices with $\Gamma(\mathcal{H}) \ge (1-\delta) \cdot \mathrm{ex}_{\Gamma}(n)$ is a member in $\mathfrak{H}$ after removing at most $\varepsilon n^r$ edges,
        \item We say $\Gamma$ is \textbf{$(\varepsilon,N_0)$-degree-stable} with respect to $\mathfrak{H}$ if every $r$-graph $\mathcal{H}$ on $n \ge N_0$ vertices with $\delta_{\Gamma}(\mathcal{H}) \ge (1-\varepsilon)\cdot \mathrm{exdeg}_{\Gamma}(n)$ is contained in $\mathfrak{H}$,
        \item We say $\Gamma$ is \textbf{$(\varepsilon,N_0)$-vertex-extendable} with respect to $\mathfrak{H}$ if for every $n \ge N_0$, 
        every $n$-vertex $r$-graph $\mathcal{H}$ with $\delta_{\Gamma}(\mathcal{H}) \ge (1-\varepsilon) \cdot \mathrm{exdeg}_{\Gamma}(n)$ and $\mathcal{H} - v \in \mathfrak{H}$ for some $v\in V(\mathcal{H})$ satisfies $\mathcal{H} \in \mathfrak{H}$, 
        \item We say $\Gamma$ is \textbf{symmetrized-stable} with respect to $\mathfrak{H}$ if every symmetrized $\Gamma$-positive $r$-graph is contained in $\mathfrak{H}$. 
     \end{enumerate}
\end{definition}
We refer the reader to Section~\ref{SEC:applications} and Fact~\ref{FACT:inj-basic-properties} for some examples of these definitions. 
Parameters $\varepsilon$ and $N_0$ will be omitted if they are not crucial in the context. 

The following simple fact follows easily from the definitions. 

\begin{fact}\label{FACT:exact-symm-increase}
    Suppose that $\Gamma$ is a symmetrization-increasing map and $\Gamma$ is symmetrized-stable with respect to $\mathfrak{H}$. 
    Then 
    \begin{align*}
        \mathrm{ex}_{\Gamma}(n)
        = \max\left\{\Gamma(\mathcal{H}) \colon \mathcal{H} \in \mathfrak{H} \text{ and } v(\mathcal{H}) = n\right\}
        \quad\text{for every}\quad n\in \mathbb{N}. 
    \end{align*}
\end{fact}

\subsection{General results}\label{SUBSEC:intro-general}
We state our main results in this subsection.

\begin{theorem}\label{THM:GeneralFunction-Stability}
     Let $r \ge 2$ be an integer, $\Gamma \colon \mathfrak{G}^{r} \to \mathbb{R}_{\ge 0}$ be a uniform, smooth, locally Lipschitz, continuous, and locally monotone map, and $\mathfrak{H}$ be a hereditary family of $r$-graphs. 
     Suppose that 
    \begin{enumerate}[label=(\roman*)]
        \item\label{THM:GeneralFunction-Stability-1} $\Gamma$ is a symmetrization-increasing, 
        \item\label{THM:GeneralFunction-Stability-2} $\Gamma$ is symmetrized-stable with respect to $\mathfrak{H}$, and
        \item\label{THM:GeneralFunction-Stability-3} $\Gamma$ is vertex-extendable with respect to $\mathfrak{H}$.
    \end{enumerate}
    Then $\Gamma$ is degree-stable with respect to $\mathfrak{H}$. 
\end{theorem}

In situations where conditions~\ref{THM:GeneralFunction-Stability-1} and~\ref{THM:GeneralFunction-Stability-2} in Theorem~\ref{THM:GeneralFunction-Stability} are not satisfied, we can bypass them if the edge-stability of $\Gamma$ has already been established, as demonstrated in the following theorem.

\begin{theorem}\label{THM:GeneralFunction-Stability-b}
    Let $r \ge 2$ be an integer, $\Gamma \colon \mathfrak{G}^{r} \to \mathbb{R}_{\ge 0}$ be a uniform, smooth, locally Lipschitz, continuous, and locally monotone map, and $\mathfrak{H}$ be a hereditary family of $r$-graphs. 
    Suppose that $\Gamma$ is edge-stable and vertex-extendable with respect to $\mathfrak{H}$.
    Then $\Gamma$ is degree-stable with respect to $\mathfrak{H}$. 
\end{theorem}

In some cases, we can establish the edge-stability of $\Gamma$ by reducing it to the degree-stability\footnote{In fact, edge-stability is sufficient for our purpose.} of a simpler map $\hat{\Gamma}$, for which we can apply Theorem~\ref{THM:GeneralFunction-Stability}. 
Such a strategy was used in works such as~\cite{MU06,PI13}.

Given two maps $\Gamma, \hat{\Gamma} \colon \mathfrak{G}^{r} \to \mathbb{R}_{\ge 0}$, we say $\hat{\Gamma}$ is a \textbf{trajectory} of $\Gamma$ if $\mathrm{ex}_{\Gamma}(n) \ge (1-o(1)) \cdot \mathrm{ex}_{\hat{\Gamma}}(n)$, and for every $\Gamma$-positive $r$-graph $\mathcal{H}$ on $n$ vertices there exists a subgraph $\mathcal{H}' \subset \mathcal{H}$ such that $|\mathcal{H}'| = |\mathcal{H}| - o(n^r)$ and $\hat{\Gamma}(\mathcal{H}') \ge \Gamma(\mathcal{H}) - o(\mathrm{ex}_{\Gamma}(n))$.

\begin{theorem}\label{THM:GeneralFunction-Stability-c}
    Let $r \ge 2$ be an integer, $\Gamma, \hat{\Gamma} \colon \mathfrak{G}^{r} \to \mathbb{R}_{\ge 0}$ be uniform, smooth, locally Lipschitz, continuous, and locally monotone maps, and $\mathfrak{H}$ be a hereditary family of $r$-graphs. 
    Suppose that 
    \begin{enumerate}[label=(\roman*)]
        \item\label{THM:GeneralFunction-Stability-c-1} $\hat{\Gamma}$ is a trajectory of $\Gamma$, 
        \item\label{THM:GeneralFunction-Stability-c-2} 
        $\hat{\Gamma}$ is symmetrization-increasing, 
        \item\label{THM:GeneralFunction-Stability-c-3} $\hat{\Gamma}$ is symmetrized-stable with respect to $\mathfrak{H}$, and 
        \item\label{THM:GeneralFunction-Stability-c-4} both $\hat{\Gamma}$ and $\Gamma$ are vertex-extendable with respect to $\mathfrak{H}$. 
    \end{enumerate}
    Then $\Gamma$ is degree-stable with respect to $\mathfrak{H}$.
\end{theorem}
\textbf{Remark.}
Conditions~\ref{THM:GeneralFunction-Stability-c-2},~\ref{THM:GeneralFunction-Stability-c-3}, and~\ref{THM:GeneralFunction-Stability-c-4} in Theorem~\ref{THM:GeneralFunction-Stability-c} can be replaced with the condition that $\hat{\Gamma}$ is edge-stable with respect to $\mathfrak{H}$ and $\Gamma$ is vertex-extendable with respect to $\mathfrak{H}$.

Proofs for Theorems~\ref{THM:GeneralFunction-Stability}, ~\ref{THM:GeneralFunction-Stability-b}, and Theorem~\ref{THM:GeneralFunction-Stability-c}
are included in Section~\ref{SEC:Proof-general}. 

\section{Applications in generalized Tur\'{a}n problems}\label{SEC:applications}
Given two $r$-graphs $Q$ and $\mathcal{H}$, a map $\phi \colon Q \to \mathcal{H}$ is a \textbf{homomorphism} if $\phi(e) \in \mathcal{H}$ for all $e \in Q$.  
When such a homomorphism exists, we say $Q$ is \textbf{$\mathcal{H}$-colorable}. 
We use $\mathrm{Hom}(Q,\mathcal{H})$ to denote the family of all homomorphisms from $Q$ to $\mathcal{H}$ and let $\mathrm{hom}(Q,\mathcal{H}) \coloneqq |\mathrm{Hom}(Q,\mathcal{H})|$. 
Let $\mathrm{Inj}(Q,\mathcal{H})$ denote the family of \textbf{injective homomorphisms} from $Q$ to $\mathcal{H}$ and let $\mathrm{inj}(Q,\mathcal{H}) \coloneqq |\mathrm{Inj}(Q,\mathcal{H})|$.
For every vertex $v\in V(\mathcal{H})$, the  \textbf{$Q$-degree} of $v$ in $\mathcal{H}$ is 
\begin{align*}
    d_{Q,\mathcal{H}}(v)
    \coloneqq \left\{\phi \in \mathrm{Inj}(Q,\mathcal{H}) \colon v\in \phi(V(Q))\right\}. 
\end{align*}
We used $d_{Q}(\mathcal{H}), \delta_{Q}(\mathcal{H}), \Delta_{Q}(\mathcal{H})$ to denote the \textbf{average, minimum, maximum} $Q$-degree of $\mathcal{H}$, respectively.

Given a family $\mathcal{F}$ of $r$-graphs, we say $\mathcal{H}$ is \textbf{$\mathcal{F}$-free}
if it does not contain any member of $\mathcal{F}$ as a subgraph.
Given an $r$-graph $Q$ and a family $\mathcal{F}$ of $r$-graphs, let 
\begin{align*}
    \mathrm{inj}(n,Q,\mathcal{F})
    \coloneqq \max\left\{\mathrm{inj}(Q,\mathcal{H}) \colon  \text{$v(\mathcal{H}) = n$ and $\mathcal{H}$ is $\mathcal{F}$-free} \right\}. 
\end{align*}
The \textbf{generalized Tur\'{a}n number} is related to $\mathrm{inj}(n,Q,\mathcal{F})$ via the definition $\mathrm{ex}(n,Q,\mathcal{F}) \coloneqq {\mathrm{inj}(n,Q,\mathcal{F})}/{|\mathrm{Aut}(Q)|}$.
Here, $\mathrm{Aut}(Q) = \mathrm{Inj}(Q,Q)$ is the automorphism group of $Q$. 
The \textbf{generalized Tur\'{a}n density} is defined as $\pi(Q, \mathcal{F}) \coloneqq \lim_{n\to \infty} {\mathrm{ex}(n,Q,\mathcal{F})}/{\binom{n}{v(Q)}}$. 
The existence of this limit can be established by a simple averaging argument used by Katona--Nemetz--Simonovits in~\cite{KNS64}. 
The \textbf{Tur\'{a}n number} of $\mathcal{F}$ is simply $\mathrm{ex}(n,K_{r}^{r},\mathcal{F})$, and for simplicity, denoted by $\mathrm{ex}(n,\mathcal{F})$. 
Similarly, the \textbf{Tur\'{a}n density} of $\mathcal{F}$ is $\pi(\mathcal{F}) \coloneqq \lim_{n\to \infty}\mathrm{ex}(n,\mathcal{F})/\binom{n}{r}$. 
A family $\mathcal{F}$ is called \textbf{$Q$-nondegenerate} (resp. \textbf{nondegenerate}) if $\pi(Q, \mathcal{F}) > 0$ (resp. $\pi(\mathcal{F}) > 0$).
In this section, we consider only nondegenerate cases. 

The study of $\mathrm{ex}(n,\mathcal{F})$ has been a central topic in Extremal Combinatorics since Tur\'{a}n's seminal work~\cite{TU41}, in which he determined the value of $\mathrm{ex}(n,K_{\ell})$ for all $\ell \ge 3$ (with the case $\ell=3$ having been addressed even earlier by Mantel~\cite{Mantel07}). 
We refer the reader to surveys~\cite{Fur91,Sid95,Keevash11} for results on Tur\'{a}n problems. 
The study of generalized Tur\'{a}n problems was initiated by Erd\H{o}s~\cite{Erdos62}, who determined $\mathrm{ex}(n,K_{k},K_{\ell})$ for all $\ell > k \ge 3$.  
Gy\H{o}ri--J\'{a}nos--Simonovits~\cite{GPS91} extended the result of Erd\H{o}s to the case where $Q$ is complete multipartite and $\mathcal{F} = K_{\ell+1}$. 
Alon--Shikhelman~\cite{AS16} initiated a systematic study on $\mathrm{ex}(n,Q, \mathcal{F})$ for graphs and obtained numerous results for various combinations of $(Q, \mathcal{F})$. 
Since then, this topic has been extensively studied in the recent decade.

To maintain consistency in notation, we will use $\mathrm{inj}(n,Q,\mathcal{F})$ instead of $\mathrm{ex}(n,Q,\mathcal{F})$ in the remaining part.
For convenience, we set $\pi_{\mathrm{inj}}(Q,\mathcal{F}) \coloneqq \lim_{n\to \infty} \frac{\mathrm{inj}(n,Q,\mathcal{F})}{n^{v(Q)}}$.

\subsection{Erd\H{o}s' pentagon problem}
In 1984, Erd\H{o}s~\cite{Erd84} conjectured that for every $n \ge 5$ the maximum number of copies of $C_5$ in an $n$-vertex triangle-free graph is attained by the balanced blowup of $C_5$. 
This conjecture was solved independently by Grzesik~\cite{Gre12} and Hatami--Hladk\'{y}--Kr\'{a}\v{l}--Norine--Razborov~\cite{HHKNR13} for large $n$, and subsequently by Lidick\'{y}--Pfender~\cite{LP18} for all $n$. 
All three proofs rely on the machinery of Flag Algebra~\cite{Raz07} and are computer-assisted. 

\begin{theorem}[\cite{Gre12,HHKNR13,LP18}]\label{THM:Pentagon-exact}
    For every $n \in \mathbb{N}$,  
    \begin{align*}
        \mathrm{inj}(n,C_5,K_3) 
        = 10\prod_{i=0}^{4}\left\lfloor\frac{n+i}{5}\right\rfloor. 
    \end{align*}
\end{theorem}

The edge-stability of the Erd\H{o}s Pentagon Problem was initially established in~\cite{HHKNR13} and subsequently refined in~\cite{PST19}.

\begin{theorem}[\cite{HHKNR13}]\label{THM:Pentagon-stability}
    For every $\varepsilon>0$ there exist $\delta>0$ and $N_0$ such that the following holds for all $n \ge N_0$. 
    If $G$ is an $n$-vertex triangle-free graph with
    $\mathrm{inj}(C_5, G) \ge (10-\delta)\left(n/5\right)^5$, 
    then $G$ is $C_5$-colorable after removing at most $\varepsilon n^2$ edges. 
\end{theorem}

As a quick application of Theorem~\ref{THM:GeneralFunction-Stability-b}, we establish the following degree-stability for the Erd\H{o}s Pentagon Problem, which strengthens both Theorems~\ref{THM:Pentagon-exact} and~\ref{THM:Pentagon-stability} for large $n$.

\begin{theorem}\label{THM:AES-pentagon}
    There exist $\varepsilon>0$ and $N_0$ such that the following holds for all $n \ge N_0$. 
    If $G$ is an $n$-vertex triangle-free graph with  $\delta_{C_5}(G) \ge (10-\varepsilon) \left(n/5\right)^4$, 
    then $G$ is $C_5$-colorable. 
\end{theorem}
An interesting problem arises in determining the tight bound for $\varepsilon$ in Theorem~\ref{THM:AES-pentagon}. 

The proof of Theorem~\ref{THM:AES-pentagon} is presented in Section~\ref{SEC:Proof-Pentagon}. 

\subsection{Blowups of $2$-covered hypergraphs versus expansions}
Given a graph $F$, the ($r$-uniform) \textbf{expansion} $H_{F}^{r}$ of $F$ is the $r$-graph obtained from $F$ by adding $r-2$ new vertices into each edge, ensuring that these $(r-2)$-sets are pairwise disjoint. 
For simplicity, let $H_{\ell+1}^{r} \coloneqq H_{K_{\ell+1}}^{r}$. 
Expansions are an important class of hypergraphs introduced by Mubayi~\cite{MU06} to provide the first explicitly defined examples that yield an infinite family of numbers realizable as Tur\'{a}n densities for hypergraphs.

Given an $r$-graph $\mathcal{G}$ on $[m]$ and pairwise disjoint sets $V_1, \ldots, V_m$, we use $\mathcal{G}(V_1, \ldots, V_m)$ to denote the $r$-graph obtained from $\mathcal{G}$ by replacing every vertex $i\in [m]$ with the set $V_i$ and every edge $\{i_1, \ldots, i_m\} \in \mathcal{G}$ with the complete $r$-partite $r$-graph with parts $V_{i_1}, \ldots, V_{i_m}$. 
The $r$-graph $\mathcal{G}(V_1, \ldots, V_m)$ is called a \textbf{blowup} of $\mathcal{G}$. 
 
We say an $r$-graph $\mathcal{G}$ is \textbf{$2$-covered} if $\partial_{r-2}\mathcal{G}$ is a complete graph, where for every $i\in [r-1]$, 
\begin{align*}
    \partial_{r-i} \mathcal{G}
    \coloneqq \left\{ e\in \binom{V(G)}{i} \colon \exists E\in \mathcal{G} \text{ such that } e\subset E \right\}. 
\end{align*}
Note that a graph is $2$-covered iff it is complete, and complete multipartite graphs are blowups of complete graphs.
Extending the result of Erd\H{o}s, Gy\H{o}ri--J\'{a}nos--Simonovits proved in~\cite{GPS91} that if $Q$ is a complete $k$-partite graph and $\ell \ge k$, then the extremal construction for the generalized Tur\'{a}n problem $\mathrm{ex}(n,Q,K_{\ell+1})$ is complete $\ell$-partite. 
The corresponding edge-stability of this problem was established very recently by Gerbner--Karim in~\cite{GK23}. 
As an application of Theorems~\ref{THM:GeneralFunction-Stability} and~\ref{THM:GeneralFunction-Stability-c}, we strengthen and extend both results to hypergraphs in the following theorem. 

Recall that a graph $F$ is \textbf{edge-critical} if there exists an edge $e\in F$ such that the chromatic number of $F\setminus\{e\}$ is strictly smaller than that of $F$.

\begin{theorem}\label{THM:Homomorphism-complete-multipartite-hypergraph}
    Let $\ell \ge k \ge r \ge  2$ be integers and $Q$ be a blowup of some $2$-covered $k$-vertex $r$-graph.
    Suppose that $F$ is an edge-critical graph with $\chi(F) = \ell+1$. 
    Then there exist $\varepsilon > 0$ and $N_0 > 0$ such that the following statements hold for all $n \ge N_0 \colon$
    \begin{enumerate}[label=(\roman*)]
        \item\label{THM:Homomorphism-complete-multipartite-hypergraph-2} Every $n$-vertex $H_{F}^{r}$-free $r$-graph $\mathcal{H}$ with $\delta_{Q}(\mathcal{H}) \ge (1-\varepsilon) \frac{v(Q)\cdot \mathrm{inj}(n,Q,H_{F}^{r})}{n}$ is $\ell$-partite. 
        \item\label{THM:Homomorphism-complete-multipartite-hypergraph-1} $\mathrm{inj}(n,Q,H_{F}^{r}) = \mathrm{inj}(Q,\mathcal{G})$ for some complete $\ell$-partite $r$-graph $\mathcal{G}$ on $n$ vertices.  
    \end{enumerate}
\end{theorem}
\textbf{Remark.}
The case where $r=2$ and $Q = K_{k}$ in Theorem~\ref{THM:Homomorphism-complete-multipartite-hypergraph}~\ref{THM:Homomorphism-complete-multipartite-hypergraph-1} corresponds to the result by Ma--Qiu in~{\cite[Corollary~1.3]{MQ20}}. 
The case where $r=2$ and $F = K_{\ell+1}$ in Theorem~\ref{THM:Homomorphism-complete-multipartite-hypergraph}~\ref{THM:Homomorphism-complete-multipartite-hypergraph-2} is a strengthen of the stability theorem by Gerbner--Karim in~{\cite[Theorem~1.1]{GK23}}.

The following improvement on Theorem~\ref{THM:Homomorphism-complete-multipartite-hypergraph} can be derived by 
combining Theorem~\ref{THM:Homomorphism-complete-multipartite-hypergraph} for the case $(r, F) = (2,K_{\ell+1})$ (or~{\cite[Theorem~1.1]{GK23}}) 
with the argument used in the proof of Theorem~\ref{THM:Turan-good-edge-critical-expansion}.  
Since Theorem~\ref{THM:Homomorphism-complete-multipartite-hypergraph} provides a nice application of Theorems~\ref{THM:GeneralFunction-Stability} and~\ref{THM:GeneralFunction-Stability-c}, we include its short proof in this paper. 

\begin{theorem}\label{THM:Homomorphism-complete-multipartite-hypergraph-b}
    Let $\ell \ge k \ge r \ge  2$ be integers and $Q$ be an $r$-graph with $\partial_{r-2} Q$ being complete $k$-partite.
    Suppose that $F$ is an edge-critical graph with $\chi(F) = \ell+1$. 
    Then there exist $\varepsilon > 0$ and $N_0 > 0$ such that the following statements hold for all $n \ge N_0 \colon$
    \begin{enumerate}[label=(\roman*)]
        \item Every $n$-vertex $H_{F}^{r}$-free $r$-graph $\mathcal{H}$ with $\delta_{Q}(\mathcal{H}) \ge (1-\varepsilon) \frac{v(Q)\cdot \mathrm{inj}(n,Q,H_{F}^{r})}{n}$ is $\ell$-partite. 
        \item $\mathrm{inj}(n,Q,H_{F}^{r}) = \mathrm{inj}(Q,\mathcal{G})$ for some complete $\ell$-partite $r$-graph $\mathcal{G}$ on $n$ vertices. 
    \end{enumerate}
\end{theorem}

The proof of Theorem~\ref{THM:Homomorphism-complete-multipartite-hypergraph} is presented in Section~\ref{SEC:proof-complete-multipartite}. 
The proof of Theorem~\ref{THM:Homomorphism-complete-multipartite-hypergraph-b} is presented in Section~\ref{SEC:Proof-Turan-goodness}. 
\subsection{Tur\'{a}n-goodness}
An interesting phenomenon, conjectured by Gerbner--Palmer~\cite{GP22} and proved by Morrison--Nir--Norin--Rz\k{a}\.{z}ewski--Wesolek~\cite{MNNRPW23}, is that for every fixed graph $Q$ and sufficiently large $n$, the extremal graph for the generalized Tur\'{a}n problem $\mathrm{ex}(n,Q,K_{\ell+1})$ is always the Tur\'{a}n graph $T(n,\ell)$, provided that $\ell$ is sufficiently large. 
Here $T(n,\ell)$ is the balanced complete $\ell$-partite graph on $[n]$. 
In fact, it was shown in~\cite{MNNRPW23} that this is true whenever $\ell \ge 300 (v(Q))^{9}$. 
The corresponding edge-stability was established recently by Gerbner--Karim in~\cite{GK23}. 
As another application of Theorem~\ref{THM:GeneralFunction-Stability-b}, we strengthen and extend both results to hypergraphs in the following theorem.

Let $T^{r}(n,\ell)$ denote the balanced complete $\ell$-partite $r$-graph on $[n]$. 

\begin{theorem}\label{THM:Turan-good-edge-critical-expansion}
    Let $r\ge 2$ be an integer and $Q$ be an $r$-graph. 
    Suppose that $F$ is an edge-critical graph with $\chi(F) = \ell+1 > 300 (v(Q))^{9}$. 
    Then there exist $\varepsilon > 0$ and $N_0 > 0$ such that the following statements hold for all integers $n \ge N_0 \colon$
    \begin{enumerate}[label=(\roman*)]
        \item\label{THM:Turan-good-edge-critical-expansion-2} Every $n$-vertex $H_{F}^{r}$-free $r$-graph $\mathcal{H}$ with $\delta_{Q}(\mathcal{H}) \ge (1-\varepsilon) \frac{v(Q) \cdot \mathrm{inj}(n,Q, H_{F}^{r})}{n}$ is $\ell$-partite. 
        \item\label{THM:Turan-good-edge-critical-expansion-1} $\mathrm{inj}(n,Q, H_{F}^{r}) = \mathrm{inj}\left(Q, T^{r}(n,\ell)\right)$. 
    \end{enumerate}
\end{theorem}
\textbf{Remark.}
In Theorem~\ref{THM:Turan-good-edge-critical-expansion}~\ref{THM:Turan-good-edge-critical-expansion-2}, one can additionally say that the ratio of the parts in $\mathcal{H}$ is close to uniform.

The proof of Theorem~\ref{THM:Turan-good-edge-critical-expansion} is presented in Section~\ref{SEC:Proof-Turan-goodness}. 
\section{Proofs for Theorems~\ref{THM:GeneralFunction-Stability},~\ref{THM:GeneralFunction-Stability-b}, and~\ref{THM:GeneralFunction-Stability-c}}\label{SEC:Proof-general}
%
Given an $n$-vertex $r$-graph $\mathcal{H}$, 
recall that two vertices $u,v \in V(\mathcal{H})$ are equivalent in $\mathcal{H}$ if $L_{\mathcal{H}}(u) = L_{\mathcal{H}}(v)$. 
Define  
\begin{align*}
    \Psi(\mathcal{H}) 
    \coloneqq \sum_{i \in [m]} |C_i|^2,   
\end{align*}
where $\{C_1, \ldots, C_{m}\}$ is the collection of equivalence classes in $\mathcal{H}$.
For a $k$-uniform map $\Gamma \colon \mathfrak{G}^{r} \to \mathbb{R}_{\ge 0}$ define for every $\mathcal{H} \in \mathfrak{G}^{r}(n)$ and for every $\delta>0$, the set 
\begin{align*}
    Z_{\Gamma, \delta}(\mathcal{H})
        \coloneqq 
        \left\{v\in V(\mathcal{H}) \colon d_{\Gamma, \mathcal{H}}(v) \le (1-\delta)\cdot \mathrm{exdeg}_{\Gamma}(n) \right\}. 
\end{align*}

We will prove the following theorem which implies Theorem~\ref{THM:GeneralFunction-Stability}. 

\begin{theorem}\label{THM:GeneralFunction-Stability-d}
     Let $\Gamma \colon \mathfrak{G}^{r} \to \mathbb{R}_{\ge 0}$ be a $k$-uniform, $(\delta_{\ast}, c_{\ast}, C_{\ast})$-smooth, locally $C$-Lipschitz, continuous, and locally monotone map with $k, \delta_{\ast}, c_{\ast},  C_{\ast}, C >0$.
     Let $\mathfrak{H}$ be a hereditary family of $r$-graphs. 
     Suppose that 
    \begin{enumerate}[label=(\roman*)]
        \item $\Gamma$ is a   symmetrization-increasing, 
        \item $\Gamma$ is symmetrized-stable with respect to $\mathfrak{H}$, and
        \item $\Gamma$ is vertex-extendable with respect to $\mathfrak{H}$.
    \end{enumerate}
    Then there exist $\varepsilon>0$ and $N_0$ such that every $r$-graph $\mathcal{H}$ on $n>N_0$ vertices with $\Gamma(\mathcal{H}) \ge (1-\varepsilon) \cdot \mathrm{ex}_{\Gamma}(n)$ satisfies $\mathcal{H}-Z_{\Gamma,\varepsilon'}(\mathcal{H}) \in \mathfrak{H}$, where $\varepsilon' \coloneqq 2\varepsilon^{\frac{c_{\ast}}{2+c_{\ast}}}$. 
\end{theorem}

Let us prove some preliminary results before presenting the proof of Theorem~\ref{THM:GeneralFunction-Stability-d}. 

\begin{fact}\label{FACT:basic-property-Lambda}
     Let $\Gamma \colon \mathfrak{G}^{r} \to \mathbb{R}_{\ge 0}$ be a $k$-uniform, $(\delta_{\ast}, c_{\ast}, C_{\ast})$-smooth, locally monotone map with $k, \delta_{\ast}, c_{\ast}, C_{\ast} >0$.
     \begin{enumerate}[label=(\roman*)]
         \item\label{FACT:basic-property-Lambda-a} If $\mathcal{H}$ is an $n$-vertex $r$-graph with $\Gamma(\mathcal{H})>0$, then for every set $S \subset V(\mathcal{H})$ we have, by the locally monotone property, that 
         \begin{align*}
             \Gamma(\mathcal{H}-S)
             \ge \Gamma(\mathcal{H}) - \sum_{v\in S}d_{\Gamma, \mathcal{H}}(v). 
         \end{align*}
         \item\label{FACT:basic-property-Lambda-c} For every $\delta \in [0,1]$, we have, by the smoothness, that 
         \begin{align*}
             \mathrm{ex}_{\Gamma}(n-\delta n) \le \mathrm{ex}_{\Gamma}(n) + o(\mathrm{ex}_{\Gamma}(n)). 
         \end{align*}
         \item\label{FACT:basic-property-Lambda-b} For every $\delta \in [0, \delta_{\ast}]$, we have, by the smoothness, that 
         \begin{align*}
             \frac{\mathrm{ex}_{\Gamma}(n)}{n}
             = (1-\delta) \cdot \frac{\mathrm{ex}_{\Gamma}(n)}{n-\delta n}
             \ge (1-\delta-o(1)) \cdot \frac{ \mathrm{ex}_{\Gamma}(n-\delta n)}{n- \delta n}. 
         \end{align*}
     \end{enumerate}
\end{fact}

The following lemma gives an upper bound for the size of $Z_{\Gamma,\delta}$ in near extremal $r$-graphs. 

\begin{lemma}\label{LEMMA:Z-delta-size}
    Let $\Gamma \colon \mathfrak{G}^{r} \to \mathbb{R}_{\ge 0}$ be a $k$-uniform, $(\delta_{\ast}, c_{\ast}, C_{\ast})$-smooth, locally $C$-Lipschitz, continuous, and locally monotone map with $k, \delta_{\ast}, c_{\ast}, C, C_{\ast} >0$ and $\delta_{\ast} \le \left(\frac{k}{2C_{\ast}}\right)^{2/c_{\ast}}$. 
    Let $\varepsilon \in (0, \delta_{\ast}^{\frac{2+c_{\ast}}{2}})$ be a real number and $n$ be sufficiently large integer. 
    Suppose that $\mathcal{H}$ is an $n$-vertex $r$-graph with $\Gamma(\mathcal{H}) \ge (1-\varepsilon) \cdot \mathrm{ex}_{\Gamma}(n)$. 
    Then for every $\delta_1 \in [2\varepsilon^{\frac{c_{\ast}}{2+c_{\ast}}}, 2\delta_{\ast}^{\frac{c_{\ast}}{2}}]$, the set $Z_{\Gamma, \delta_1}(\mathcal{H})$ has size at most $(\delta_1/2)^{\frac{2}{c_{\ast}}} n$. 
\end{lemma}
\begin{proof}
    Fix $\delta_1 \in [2\varepsilon^{\frac{c_{\ast}}{2+c_{\ast}}}, 2\delta_{\ast}^{\frac{c_{\ast}}{2}}]$. 
    Let $\delta_2 \coloneqq (\delta_1/2)^{\frac{2}{c_{\ast}}}$. 
    Simple calculations show that $\delta_1 \ge 2\delta_2^{\frac{c_{\ast}}{2}}$, $k\delta_1 \delta_2/2 \ge \varepsilon$, and $\delta_2 \le \delta_{\ast}$. 
    Let $V\coloneqq V(\mathcal{H})$, $Z\coloneqq Z_{\Gamma, \delta_1}(\mathcal{H})$, and $U \coloneqq V\setminus Z$. 
    Suppose to the contrary that $|Z| \ge \delta_2 n$. 
    Let $Z'\subset Z$ be a subset of size $\delta_2 n$. 
    It follows from Fact~\ref{FACT:basic-property-Lambda}~\ref{FACT:basic-property-Lambda-a} that 
    \begin{align*}
        \mathrm{ex}_{\Gamma}(n-\delta_2 n)
        \ge \Gamma(\mathcal{H}[U])
        & > \Gamma(\mathcal{H}) - \delta_{2}n \times (1-\delta_1)k \cdot \frac{\mathrm{ex}_{\Gamma}(n)}{n} \\
        & = \Gamma(\mathcal{H}) - \delta_{2}k \cdot \mathrm{ex}_{\Gamma}(n) + \delta_2 \delta_1 k \cdot \mathrm{ex}_{\Gamma}(n) \\
        & \ge \mathrm{ex}_{\Gamma}(n)- \delta_{2}k \cdot \mathrm{ex}_{\Gamma}(n) + \left(\delta_1 \delta_2 k - \varepsilon\right) \cdot \mathrm{ex}_{\Gamma}(n) \\
        & \ge \mathrm{ex}_{\Gamma}(n)- \delta_{2}k \cdot \mathrm{ex}_{\Gamma}(n) + \delta_2^{1+\frac{c_{\ast}}{2}} k \cdot \mathrm{ex}_{\Gamma}(n) \\
        & \ge \mathrm{ex}_{\Gamma}(n)- \delta_{2}k \cdot \mathrm{ex}_{\Gamma}(n) + 2C_{\ast} \delta_2^{1+c_{\ast}} \cdot \mathrm{ex}_{\Gamma}(n)
    \end{align*}
    contradicting the smoothness of $\Gamma$. 
\end{proof}

\begin{lemma}\label{LEMMA:Lambda-set-extendability}
    Let $\Gamma \colon \mathfrak{G}^{r} \to \mathbb{R}_{\ge 0}$ be a $k$-uniform, $(\delta_{\ast}, c_{\ast}, C_{\ast})$-smooth, locally $C$-Lipschitz, continuous, and locally monotone map with $k, \delta_{\ast}, c_{\ast}, C, C_{\ast} >0$. 
    Let $\mathfrak{H}$ be a hereditary family of $r$-graphs. 
    Suppose that $\varepsilon>0$ is a sufficiently small constant and $\Gamma$ is $(3\varepsilon, N_0)$-vertex-extendable with respects to $\mathfrak{H}$.
    Then the following holds for sufficiently large $n \colon$
    If an $n$-vertex $r$-graph $\mathcal{H}$ with $\delta_{\Gamma}(\mathcal{H}) \ge (1-\varepsilon)\cdot \mathrm{exdeg}_{\Gamma}(n)$ satisfies $\mathcal{H}-S \in \mathfrak{H}$ for some set $S\subset V(\mathcal{H})$ of size at most  $\varepsilon n/C$, then $\mathcal{H} \in \mathfrak{H}$.
\end{lemma}
\begin{proof}
    Choose a minimal set $S' \subset S$ with $\mathcal{H}-S' \in \mathfrak{H}$. 
    If $S' = \emptyset$, then we are done. 
    So we may assume that $S' \neq \emptyset$. 
    Fix $v\in S'$. 
    Let $S''  \coloneqq S'\setminus\{v\}$ and $\mathcal{H}'' \coloneqq \mathcal{H} - S''$.  Notice that 
    $v(\mathcal{H}'') \ge n - |S|  \ge (1-\varepsilon/C)n \ge n/2 \ge N_0$ and by the local $C$-Lipschitz assumption,  
    \begin{align*}
        \delta(\mathcal{H}'')
        & \ge \delta(\mathcal{H}) -  \frac{C |S|k \cdot \mathrm{ex}_{\Gamma}(n)}{n^2} - o(\mathrm{exdeg}_{\Gamma}(n)) \\
        & \ge (1-\varepsilon)k \cdot \frac{\mathrm{ex}_{\Gamma}(n)}{n} - \frac{C |S|k \cdot \mathrm{ex}_{\Gamma}(n)}{n^2} - o(\mathrm{exdeg}_{\Gamma}(n)) \\
        & > (1-2\varepsilon)k \cdot \frac{\mathrm{ex}_{\Gamma}(n)}{n}
        > (1-3\varepsilon)k \cdot \frac{\mathrm{ex}_{\Gamma}(n-|S''|)}{n-|S''|}. 
    \end{align*}
    Here, the last inequality follows from Fact~\ref{FACT:basic-property-Lambda}~\ref{FACT:basic-property-Lambda-b} and the assumption that $\varepsilon$ is sufficiently small. 
    Since $\mathcal{H}''-v = \mathcal{H}'- S' \in \mathfrak{H}$, it follows from the vertex-extendability of $\Gamma$ that $\mathcal{H}'' \in \mathfrak{H}$, contradicting the minimality of $S'$. 
\end{proof}

Now we are ready to prove Theorem~\ref{THM:GeneralFunction-Stability}. 
\begin{proof}[Proof of Theorem~\ref{THM:GeneralFunction-Stability}]
    Let $\Gamma \colon \mathfrak{G}^{r} \to \mathbb{R}_{\ge 0}$ be a $k$-uniform, $(\delta_{\ast}, c_{\ast}, C_{\ast})$-smooth, locally $C$-Lipschitz, continuous, and locally monotone map with $k, \delta_{\ast}, c_{\ast}, C, C_{\ast} >0$. We may assume that $c_{\ast} \le 1$. 
    Choose $\varepsilon >0$ sufficiently small and $N_1$ sufficiently large such that  that $\Gamma$ is$(6\varepsilon^{\frac{c_{\ast}}{2+c_{\ast}}}, N_1)$-vertex-extendable with respects to $\mathfrak{H}$. 
    Fix $n \ge N_1$ and assume that the conclusion fails for some $n$-vertex $r$-graph. 
    We pick a counterexample $\mathcal{H}$ with $v(\mathcal{H}) = n$ such that the pair $\left(\Gamma(\mathcal{H}), \Psi(\mathcal{H})\right)$ is lexicographically maximal (this is possible since $n$ is fixed). 
    Let $\{C_1, \ldots, C_m\}$ be the collection of equivalence classes in $\mathcal{H}$. 

    Let $\delta_1 \coloneqq 2\varepsilon^{\frac{c_{\ast}}{2+c_{\ast}}}$ and $\delta_2 \coloneqq \varepsilon^{\frac{2}{2+c_{\ast}}}$. 
    Since $c_{\ast} \le 1$, we may choose $\varepsilon$ small enough in the beginning such that $\delta_2 \ll \delta_1/C$. 
    Let $Z \coloneqq Z_{\Gamma,\delta_1}(\mathcal{H})$. 
    Since $\mathcal{H}$ is a counterexample, we have $\mathcal{H}-Z\not\in \mathfrak{H}$. 
    Consequently, $\mathcal{H} \not\in \mathfrak{H}$ as well (since $\mathfrak{H}$ is hereditary). 
    This means that $\mathcal{H}$ is not symmetrized. 
    By relabelling the equivalence classes we may assume that $\partial_{r-2}\mathcal{H}$ has no edge between $C_1$ and $C_2$. 
    By symmetry, we may assume that $|C_1| \le |C_2|$. 
    Choose two arbitrary vertices $v_1 \in C_1$ and $v_2 \in C_2$. 
    Define 
    \begin{align*}
        \mathcal{H}' 
        \coloneqq 
        \begin{cases}
            \mathcal{H}_{v_1 \to v_2}, & \quad\text{if}\quad \Gamma(\mathcal{H}_{v_1 \to v_2}) \ge \Gamma(\mathcal{H}), \\
            \mathcal{H}_{v_2 \to v_1}, & \quad\text{otherwise}. 
        \end{cases}
    \end{align*}
    If $\Gamma(\mathcal{H}_{v_1 \to v_2}) > \Gamma(\mathcal{H})$, then we have $\left(\Gamma(\mathcal{H}'), \Psi(\mathcal{H}')\right) >_{\mathrm{lex}} \left(\Gamma(\mathcal{H}), \Psi(\mathcal{H})\right)$.  
    If $\Gamma(\mathcal{H}_{v_1 \to v_2}) = \Gamma(\mathcal{H})$, then we have 
    \begin{align*}
        \Psi(\mathcal{H}') - \Psi(\mathcal{H})
        \ge \left(|C_1|-1\right)^2 + \left(|C_2|+1\right)^2 - |C_1|^2 - |C_2|^2
        = 2\left(|C_2|-|C_1|+1\right)
        \ge 2, 
    \end{align*}
    and hence, $\left(\Gamma(\mathcal{H}'), \Psi(\mathcal{H}')\right) >_{\mathrm{lex}} \left(\Gamma(\mathcal{H}), \Psi(\mathcal{H})\right)$. 
    If $\Gamma(\mathcal{H}_{v_1 \to v_2}) < \Gamma(\mathcal{H})$, then it follows from the symmetrization-increasing assumption that $\Gamma(\mathcal{H}_{v_2 \to v_1}) > \Gamma(\mathcal{H})$, and consequently,   $\left(\Gamma(\mathcal{H}'), \Psi(\mathcal{H}')\right) >_{\mathrm{lex}} \left(\Gamma(\mathcal{H}), \Psi(\mathcal{H})\right)$.
    In all cases, we have $\left(\Gamma(\mathcal{H}'), \Psi(\mathcal{H}')\right) >_{\mathrm{lex}} \left(\Gamma(\mathcal{H}), \Psi(\mathcal{H})\right)$.
    Therefore, by the maximality of $\left(\Gamma(\mathcal{H}), \Psi(\mathcal{H})\right)$, we have $\mathcal{H}' - Z_{\Gamma,\delta_1}(\mathcal{H}') \in \mathfrak{H}$.
    
    Let $Q \coloneqq Z_{\Gamma,\delta_1}(\mathcal{H}') \cup \{v_1\}$ and note that $\mathcal{H}-Q = \mathcal{H}' - Q \in \mathfrak{H}$. 
    By Lemma~\ref{LEMMA:Z-delta-size}, we have $|Q| \le \delta_2 n + 1 \le 2\delta_2 n$.
    Hence, it follows from the local $C$-Lipschitz assumption that 
    \begin{align*}
        \delta_{\Gamma}(\mathcal{H}\setminus(Q\cap Z)) 
        & \ge (1-\delta_1)k \cdot \frac{\mathrm{ex}_{\Gamma}(n)}{n} - 2C \delta_2 n  \cdot \frac{\mathrm{ex}_{\Gamma}(n)}{n} - o(\mathrm{exdeg}_{\Gamma}(n)) \\
        & \ge (1-2\delta_1)k \cdot \frac{\mathrm{ex}_{\Gamma}(n)}{n} 
         \ge (1-3\delta_1)k \cdot \frac{\mathrm{ex}_{\Gamma}(n-|Q\cap Z|)}{n-|Q\cap Z|}. 
    \end{align*}
    Therefore, it follows from Lemma~\ref{LEMMA:Lambda-set-extendability}  that $\mathcal{H}-(Q\cap Z) \in \mathfrak{H}$, contradicting $\mathcal{H} - Z \not\in \mathfrak{H}$.
\end{proof}

Next, we prove Theorem~\ref{THM:GeneralFunction-Stability-b}. 
\begin{proof}[Proof of Theorem~\ref{THM:GeneralFunction-Stability-b}]
    Let $\Gamma \colon \mathfrak{G}^{r} \to \mathbb{R}_{\ge 0}$ be a $k$-uniform, $(\delta_{\ast}, c_{\ast}, C_{\ast})$-smooth, locally $C$-Lipschitz, continuous, and locally monotone map with $k, \delta_{\ast}, c_{\ast}, C, C_{\ast} >0$. 
    We may assume that $c_{\ast} \le 1$. 
    Let $\mathfrak{H}$ be a hereditary family of $r$-graphs. 
    Fix $0 < \varepsilon_1 \ll \varepsilon_2 \ll \varepsilon_3 \ll \delta_{\ast}$ sufficiently small and $N_1$ to be sufficiently large such that 
    \begin{enumerate}[label=(\roman*)]
        \item\label{assum:deg-stab-1}  every $r$-graph $\mathcal{H}$ on $n \ge N_1$ vertices with $\Gamma(\mathcal{H}) \ge (1-\varepsilon_1) \cdot \mathrm{ex}_{\Gamma}(n)$ is a member in $\mathfrak{H}$ after removing at most $\varepsilon_2 n^r/2$ edges,  
        \item\label{assum:deg-stab-2} $\Gamma$ is $(6\varepsilon_3^{\frac{c_{\ast}}{2+c_{\ast}}}, N_1)$-vertex-extendable with respect to $\mathfrak{H}$, and 
        \item\label{assum:deg-stab-3} every $r$-graph $\mathcal{H}$ on $n \ge N_1$ vertices satisfies 
        \begin{align*}
            \left| \sum_{v\in V(\mathcal{H})}d_{\Gamma, \mathcal{H}}(v) - k \cdot \Gamma(\mathcal{H}) \right|
               \le \frac{\varepsilon_1}{2} \cdot \mathrm{ex}_{\Gamma}(n).
        \end{align*}
    \end{enumerate}
    %
    Let $\mathcal{H}$ be an $r$-graph on $n \ge 2N_1$ vertices with $\delta_{\Gamma}(\mathcal{H}) \ge \left(1 - \varepsilon_1/2\right)  \cdot 
 \mathrm{exdeg}_{\Gamma}(n)$. 
    We aim to show that $\mathcal{H} \in \mathfrak{H}$.

    Let $V\coloneqq V(\mathcal{H})$. 
    First, notice from~\ref{assum:deg-stab-3} above that 
    \begin{align*}
        \Gamma(\mathcal{H})
        &\ge \frac{1}{k} \sum_{v\in V}d_{\Gamma, \mathcal{H}}(v) -  \frac{\varepsilon_1}{2} \cdot \mathrm{ex}_{\Gamma}(n) \\
        & \ge \frac{1}{k} \cdot n \cdot \left(1 - \frac{\varepsilon_1}{2}\right) k  \cdot \frac{\mathrm{ex}_{\Gamma}(n)}{n} -  \frac{\varepsilon_1}{2} \cdot \mathrm{ex}_{\Gamma}(n)
         \ge (1-\varepsilon_1) \cdot \mathrm{ex}_{\Gamma}(n).  
    \end{align*}
    So it follows from~\ref{assum:deg-stab-1} above that there exists a subgraph $\mathcal{H}_1 \subset \mathcal{H}$ with $\mathcal{H}_1 \in \mathfrak{H}$ and $|\mathcal{H}_1| \ge |\mathcal{H}| - \varepsilon_2 n^r/2$. 
    Since $\Gamma$ is continuous, we have 
    \begin{align*}
        \Gamma(\mathcal{H}_1)
         \ge \Gamma(\mathcal{H}) - \frac{\varepsilon_3}{2} \cdot  \mathrm{ex}_{\Gamma}(n) 
         \ge (1-\varepsilon_3) \cdot \mathrm{ex}_{\Gamma}(n). 
    \end{align*}
    Let 
    \begin{align*}
        \delta_1 \coloneqq 2\varepsilon_3^{\frac{c_{\ast}}{2+c_{\ast}}}, \quad 
        \delta_2 \coloneqq \varepsilon_3^{\frac{2}{2+c_{\ast}}}, \quad 
        Z \coloneq Z_{\Gamma, \delta_1}(\mathcal{H}_1), 
        \quad \text{and}\quad 
        U \coloneq V\setminus Z. 
    \end{align*}
    Since $c_{\ast} \le 1$, by taking $\varepsilon_3$ to be sufficiently small in the beginning we can guarantee that $\delta_2 \ll \delta_1/C$. 
    It follows from Lemma~\ref{LEMMA:Z-delta-size} that $|Z| \le \delta_2 n$. 
    Hence, by the locally $C$-Lipschitz assumption, we obtain 
    \begin{align*}
        \delta_{\Gamma}(\mathcal{H}_1[U])
        & \ge (1-\delta_1) \cdot \mathrm{exdeg}_{\Gamma}(n) - \frac{Ck \cdot \mathrm{ex}_{\Gamma}(n)}{n^2} |Z| - o(\mathrm{exdeg}_{\Gamma}(n))\\
        & \ge (1-\delta_1) \cdot \mathrm{exdeg}_{\Gamma}(n) - 2 C \delta_2 \cdot \mathrm{exdeg}_{\Gamma}(n) 
        \ge (1-2\delta_1) \cdot \mathrm{exdeg}_{\Gamma}(n). 
    \end{align*}
    Since $\mathfrak{H}$ is hereditary and $\mathcal{H}_1 \in \mathfrak{H}$, we have $\mathcal{H}_1[U] \in \mathfrak{H}$ as well. 
    Let $v_1, \ldots, v_n$ be an ordering of vertices in $V$ such that $\{v_1, \ldots, v_{|U|}\} = U$. 
    For convenience, let $V_i \coloneq \{v_1, \ldots, v_i\}$ for $i\in [n]$. 
    Let $\mathcal{G}_0 \coloneqq \mathcal{H}_1[U]$, and for $i\in [n]$ let $\mathcal{G}_i \coloneqq \mathcal{G}_{i-1} \cup \{e\in \mathcal{H}[V_i] \colon v_i\in e\}$. 
    Note that $\mathcal{H}_1[U] = \mathcal{G}_0 \subset \cdots \subset \mathcal{G}_n = \mathcal{H}$. Moreover, $V(\mathcal{G}_i) = U$ for $i\le |Z|$, and $V(\mathcal{G}_i) = V_i$ for $i \ge |Z|$. 
    
    We prove by induction on $i$ that $\mathcal{G}_i \in \mathfrak{H}$. 
    The base case $i=0$ is clear, so we may assume that $i\ge 1$. 
    Assume that we have shown that $G_{i-1} \in \mathfrak{H}$ for some $i\ge 1$, and we want to show that $\mathcal{G}_i \coloneqq \mathcal{G}_{i-1} \cup \{e\in \mathcal{H}[V_i] \colon v_i\in e\}$ is also contained in $\mathfrak{H}$. 
    
    If $v_i \in U$,  
    then by the locally monotone property of $\Gamma$, we have 
    \begin{align*}
        \delta_{\Gamma}(\mathcal{G}_i) 
        \ge \delta_{\Gamma}(\mathcal{H}_1[U]) 
        \ge (1-2\delta_1) \cdot \mathrm{exdeg}_{\Gamma}(n)
        \ge (1-3\delta_1) \cdot \mathrm{exdeg}_{\Gamma}(n-|Z|), 
    \end{align*}
    It follows from $\mathcal{G}_{i} - v_i \subset \mathcal{G}_{i-1} \in \mathfrak{H}$ 
    and~\ref{assum:deg-stab-2} above that $\mathcal{G}_i \in \mathfrak{H}$. 
    
    If $v_i\in Z$, 
    then by the locally $C$-Lipschitz assumption, we have 
    \begin{align*}
        \delta_{\Gamma}(\mathcal{G}_{i}) 
        & \ge \delta_{\Gamma}(\mathcal{H}) - \frac{Ck \cdot \mathrm{ex}_{\Gamma}(n)}{n^2} |Z| - o(\mathrm{exdeg}_{\Gamma}(n)) \\
        & \ge (1-\delta_1) \cdot \mathrm{exdeg}_{\Gamma}(n) - 2 C \delta_2 \cdot \mathrm{exdeg}_{\Gamma}(n) \ge (1-2\delta_1) \cdot \mathrm{exdeg}_{\Gamma}(n). 
    \end{align*}
    Similarly, it follows from $\mathcal{G}_{i} - v_i \subset \mathcal{G}_{i-1} \in \mathfrak{H}$
    and~\ref{assum:deg-stab-2} above that $\mathcal{G}_i \in \mathfrak{H}$. 
    This proves our claim. 
    Therefore, $\mathcal{H} = \mathcal{G}_n \in \mathfrak{H}$. 
\end{proof}

Finally, we present the proof of Theorem~\ref{THM:GeneralFunction-Stability-c}. 

\begin{proof}[Proof of Theorem~\ref{THM:GeneralFunction-Stability-c}]
    It suffices to prove that $\Gamma$ is edge-stable with respect to $\mathfrak{H}$, since it would follow from Assumption~\ref{THM:GeneralFunction-Stability-c-4} and Theorem~\ref{THM:GeneralFunction-Stability-b} that $\Gamma$ is degree-stable with respect to $\mathfrak{H}$. 
    
    Let $\mathcal{H}$ be an $r$-graph on $n$ vertices with $\Gamma(\mathcal{H}) = (1-o(1))\mathrm{ex}_{\Gamma}(n)$.
    It follows from the definition of trajectory that there exists a subgraph $\mathcal{H}' \subset \mathcal{H}$ such that $|\mathcal{H}'| = |\mathcal{H}| - o(n^r)$ and $\hat{\Gamma}(\mathcal{H}') = \Gamma(\mathcal{H}) - o(\mathrm{ex}_{\Gamma}(n)) = (1-o(1))\mathrm{ex}_{\Gamma}(n) \ge (1-o(1))\mathrm{ex}_{\hat{\Gamma}}(n)$. 
    By Assumptions~\ref{THM:GeneralFunction-Stability-c-2},~\ref{THM:GeneralFunction-Stability-c-3},~\ref{THM:GeneralFunction-Stability-c-4} and Theorem~\ref{THM:GeneralFunction-Stability}, we know that $\hat{\Gamma}$ is degree-stable with respect to $\mathfrak{H}$. In particular, $\hat{\Gamma}$ is edge-stable with respect to $\mathfrak{H}$. 
    Therefore, there exists $\mathcal{H}'' \subset \mathcal{H}'$ such that $|\mathcal{H}''| = |\mathcal{H}'| - o(n^r) = |\mathcal{H}| - o(n^r)$ and $\mathcal{H}'' \in \mathfrak{H}$. 
    This proves that $\Gamma$ is edge-stable with respect to $\mathfrak{H}$. 
\end{proof}

\section{Preparations for generalized Tur\'{a}n problems}\label{SEC:Prelim}
In this section, 
we present some definitions and useful facts related to generalized Tur\'{a}n problems. 

Recall that an $r$-graph $\mathcal{H}$ is \textbf{$\mathcal{G}$-colorable} if there exists a homomorphism from $\mathcal{H}$ to $\mathcal{G}$. 
Extending the definition of the chromatic number for graphs, the \textbf{chromatic number} $\chi(Q)$ of an $r$-graph $Q$ is the minimum integer $\ell$ such that $Q$ is $K_{\ell}^r$-colorable. 
For integers $\ell \ge r \ge 2$, we use $\mathfrak{K}_{\ell}^{r}$ to denote the collection of all $K_{\ell}^{r}$-colorable $r$-graphs. 
Note that $\mathfrak{K}_{\ell}^{r}$ is a hereditary family. 

Given a partition $V_1 \cup \cdots \cup V_{\ell} = V$, we use $K^{r}(V_1, \ldots, V_{\ell})$ to denote the complete $\ell$-partite $r$-graph with parts $V_1, \ldots, V_{\ell}$. 
Recall that $T^{r}(n,\ell)$ is the balanced complete $\ell$-partite $r$-graph on $n$ vertices. 

Given an $r$-graph $\mathcal{H}$ and a set $S\subset V(\mathcal{H})$, we use $\mathcal{H}[S]$ to denote the \textbf{induced subgraph} of $\mathcal{H}$ on $S$, and use $\mathcal{H} - S$ to denote the induced subgraph of $\mathcal{H}$ on $V(\mathcal{H})\setminus S$. 
For every vertex $v \in V(\mathcal{H})$ the \textbf{neighborhood} of $v$ in $\mathcal{H}$ is defined as 
\begin{align*}
    N_{\mathcal{H}}(v)
    \coloneqq 
    \left\{u\in V(\mathcal{H}) \colon \exists e\in \mathcal{H} \text{ such that } \{u,v\} \subset e \right\}.
\end{align*}
For every $k \in \mathbb{R}^{+}$ let 
\begin{align*}
    N_{\mathcal{H}}^{k}(v)
    \coloneqq 
    \left\{u\in V(\mathcal{H}) \colon d_{\mathcal{H}}(uv)  \ge  kr \binom{n}{r-3} \right\}.
\end{align*}
A simple observation is that for every $u\in N_{\mathcal{H}}^{k}(v)$ the link $L_{\mathcal{H}}(uv)$ contains at least $k$ pairwise disjoint sets. 

An $r$-graph $Q$ is \textbf{non-empty} if its edge set is not empty. 
A vertex $v\in V(Q)$ is \textbf{isolated} if there is no edge in $Q$ containing $v$. 
For the remainder of this paper, we can assume that $Q$ does not contain any isolated vertices, owing to the following identity$\colon$
\begin{align*}
    \mathrm{inj}(Q,\mathcal{H})
    = (n-v(Q)+1) \cdot \mathrm{inj}(Q-v,\mathcal{H})
\end{align*}
where $v\in V(Q)$ is an isolated vertex.

Given a family $\mathcal{F}$ of $r$-graphs, the \textbf{$\mathcal{F}$-freeness indicator map} $\mathbbm{1}_{\mathcal{F}} \colon \mathfrak{G}^{r} \to \{0,1\}$ is defined as 
\begin{align*}
    \mathbbm{1}_{\mathcal{F}}(\mathcal{H})
    \coloneqq 
    \begin{cases}
        1, & \quad\text{if $\mathcal{H}$ is $\mathcal{F}$-free}, \\
        0, & \quad \text{otherwise}. 
    \end{cases}
\end{align*}

A family $\mathcal{F}$ of $r$-graphs is \textbf{blowup-invariant} if every blowup of every $\mathcal{F}$-free $r$-graph remains $\mathcal{F}$-free. 

\begin{fact}\label{FACT:blowup-invariant-symmetrization-increasing}
    Suppose that $\mathcal{F}$ is a blowup-invariant family. 
    Then the $\mathcal{F}$-freeness indicator map $\mathbbm{1}_{\mathcal{F}}$ is symmetrization-increasing. 
    In particular, if $\Gamma \colon \mathfrak{G}^{r} \to \mathbb{R}$ is symmetrization-increasing, then the map $\tilde{\Gamma}$ defined by $\tilde{\Gamma}(\mathcal{H}) \coloneqq \Gamma(\mathcal{H}) \cdot \mathbbm{1}_{\mathcal{F}}(\mathcal{H})$ is also symmetrization-increasing. 
\end{fact}

\begin{fact}[see e.g.~\cite{KNS64}]\label{FACT:limit-exist-generalized-Turan}
    Let $r \ge 2$ be an integer and $\mathcal{F}$ be a family of $r$-graphs.  
    For every $r$-graph $Q$, the limit $\lim\limits_{n\to \infty} \frac{\mathrm{inj}(n,Q,\mathcal{F})}{n^{v(Q)}}$ exists.    
\end{fact}

The following simple observation is crucial for the proof of Theorem~\ref{THM:Turan-good-edge-critical-expansion}.

\begin{fact}\label{FACT:homomorphism-shadow}
    Let $r > i \ge 1$ be integers.
    For every pair $Q, \mathcal{H}$ of $r$-graphs, it holds that 
    \begin{align*}
        \mathrm{inj}(Q, \mathcal{H})
        \le \mathrm{inj}(\partial_{r-i}Q, \partial_{r-i}\mathcal{H}). 
    \end{align*}
    Moreover, if $\mathcal{H}$ is complete multipartite, then 
    \begin{align*}
        \mathrm{inj}(Q, \mathcal{H})
        = \mathrm{inj}(\partial_{r-i}Q, \partial_{r-i}\mathcal{H}). 
    \end{align*}
\end{fact}

\begin{fact}\label{FACT:inj-Hom}
    For every $r$-graph $Q$ and $n$-vertex $r$-graph $\mathcal{H}$ the following statements hold.  
    \begin{enumerate}[label=(\roman*)]
        \item $\mathrm{hom}(Q,\mathcal{H}) =\mathrm{inj}(Q,\mathcal{H}) + O(n^{v(Q)-1})$. 
        \item 
        $\sum_{v\in V(\mathcal{H})} d_{Q,\mathcal{H}}(v)
            = v(Q)\cdot \mathrm{inj}(Q,\mathcal{H})$, 
        and hence, $d_{Q}(\mathcal{H})
        = \frac{v(Q)\cdot \mathrm{inj}(Q,\mathcal{H})}{n}$. 
    \end{enumerate}
\end{fact}

\begin{fact}\label{FACT:inj-basic-properties}
    Let $r \ge 2$ be an integer, $Q$ be an $r$-graph, and $\mathcal{F}$ be a family of $r$-graphs. Suppose that $\pi_{\mathrm{inj}}(Q, \mathcal{F}) > 0$. Then the following statements hold for the map $\Gamma \colon \mathfrak{G}^{r} \to \mathbb{R}$ defined by $\Gamma(\mathcal{H})
        \coloneqq \mathrm{inj}(Q, \mathcal{H}) \cdot \mathbbm{1}_{\mathcal{F}}(\mathcal{H})$.  
    \begin{enumerate}[label=(\roman*)]
        \item $\Gamma$ is $v(Q)$-uniform. 
        \item $\mathrm{ex}_{\Gamma}(n) =\left(\pi_{\mathrm{inj}}(Q, \mathcal{F}) + o(1) \right) n^{v(Q)}$, and hence, $\Gamma$ is $\left(1, 1, 2^{v(Q)}\right)$-smooth.
        \item For every $n$-vertex $\Gamma$-positive $r$-graph $\mathcal{H}$, for every $U\subset V(\mathcal{H})$, and for every $v\in U$ we have 
        \begin{align*}
            d_{\Gamma, \mathcal{H}[U]}(v)
            & \ge d_{\Gamma, \mathcal{H}}(v) 
                -  v(Q)\left(n-|U|\right) n^{v(Q)-2} \\
            & = d_{\Gamma, \mathcal{H}}(v) 
                - (1+o(1))\frac{\left(v(Q)\right)^2}{\pi_{\mathrm{inj}}(Q,\mathcal{F})} \frac{\mathrm{exdeg}_{\Gamma}(n)}{n}\left(n-|U|\right), 
        \end{align*}
        and hence, $\Gamma$ is locally $\frac{2\left(v(Q)\right)^2}{\pi_{\mathrm{inj}}(Q,\mathcal{F})}$-Lipschitz.
        \item For every $n$-vertex $r$-graph $\mathcal{H}$ and for every subgraph $\mathcal{H}' \subset \mathcal{H}$ we have 
        \begin{align*}
            \Gamma(\mathcal{H}')
            \ge \Gamma(\mathcal{H}) 
                - r!|Q| |\mathcal{H}\setminus \mathcal{H}'| n^{v(Q)-r}
            \ge \Gamma(\mathcal{H}) 
                - (1+o(1))\frac{r!|Q|}{\pi_{\mathrm{inj}}(Q,\mathcal{F})} \frac{|\mathcal{H}\setminus \mathcal{H}'|}{n^r}  \mathrm{ex}_{\Gamma}(n), 
        \end{align*}
        and hence, $\Gamma$ is continuous.
        \item For every $n$-vertex $\Gamma$-positive $r$-graph $\mathcal{H}$, for every subgraph $\mathcal{H}' \subset \mathcal{H}$, and for every $v\in V(\mathcal{H}')$, it is clear that $d_{\Gamma, \mathcal{H}'}(v) \le d_{\Gamma, \mathcal{H}}(v)$, and hence, $\Gamma$ is locally monotone. 
    \end{enumerate}
\end{fact}

Let $Q$ be an $r$-graph, and let $\mathcal{H}$ is an $r$-graph on $[n]$. 
Inspired by the concept of the Lagrangian for graphs and hypergraphs~\cite{MS65,FR84},  
we define the \textbf{$Q$-Lagrange polynomial} of $\mathcal{H}$ as follows  
\begin{align*}
    P_{Q,\mathcal{H}}(X_1, \ldots, X_{n})
    \coloneqq \sum_{\phi \in \mathrm{Hom}(Q,\mathcal{H})} \prod_{i\in \phi(V(Q))}X_i. 
\end{align*}
The \textbf{$Q$-Lagrangian} of $\mathcal{H}$ is 
\begin{align*}
    \lambda_{Q}(\mathcal{H})
    \coloneqq \max\left\{P_{Q,\mathcal{H}}(x_1, \ldots, x_{n}) \colon (x_1, \ldots, x_{n}) \in \Delta^{n-1} \right\}. 
\end{align*}
A vector $(x_1, \ldots, x_n) \in \Delta^{n-1}$ is called \textbf{$(Q,\mathcal{H})$-optimal} if $P_{Q,\mathcal{H}}(x_1, \ldots, x_{n}) = \lambda_{Q}(\mathcal{H})$. 
For convenience, let $\mathrm{Opt}(Q, \mathcal{H})\subset \Delta^{n-1}$ be the collection of all $(Q,\mathcal{H})$-optimal vectors. 
Notice that $\mathrm{Opt}(Q, \mathcal{H})$ is a closed set. 

A very useful property of the $Q$-Lagrangian is demonstrated by the following fact.  

\begin{fact}\label{FACT:Lagrangian-meaning}
    Let $Q$ be an $r$-graph and $\mathcal{H}$ be an $r$-graph on $[m]$.
    Suppose that $V_1\cup \cdots \cup V_{m} = [n]$ is a partition and $x_i \coloneqq |V_i|/n$ for $i\in [m]$. 
    Then 
    \begin{align*}
        \mathrm{inj}(Q,\mathcal{G}(V_1, \ldots, V_m))
        = P_{Q,\mathcal{H}}(x_1, \ldots, x_{m}) \cdot n^{v(Q)} + O(n^{v(Q)-1}). 
    \end{align*}
    Moreover, if $Q$ is $2$-covered, then 
    \begin{align*}
        \mathrm{inj}(Q,\mathcal{G}(V_1, \ldots, V_m))
        = P_{Q,\mathcal{H}}(x_1, \ldots, x_{m}) \cdot n^{v(Q)}.  
    \end{align*}
    In particular, $\mathrm{inj}(Q,\mathcal{G}(V_1, \ldots, V_m)) \le \left(\lambda_{Q}(\mathcal{G}) + o(1)\right) \cdot n^{v(Q)}$. 
\end{fact}

Observe that if two families $\mathcal{F}, \hat{\mathcal{F}}$ satisfy $\hat{\mathcal{F}} \subset \mathcal{F}$, then $\pi_{\mathrm{inj}}(Q, \mathcal{F}) \le \pi_{\mathrm{inj}}(Q,\hat{\mathcal{F}})$ for every $r$-graph $Q$. 
In general, 
given two families $\mathcal{F}, \hat{\mathcal{F}}$ of $r$-graphs, we say $\hat{\mathcal{F}} \le_{\mathrm{hom}} \mathcal{F}$ if for every $\hat{F} \in \hat{\mathcal{F}}$ there exists $F\in \mathcal{F}$ such that $F$ is $\hat{F}$-colorable. 
For example, $\{K_4^{3-}, F_5\} \le_{\mathrm{hom}} \{F_5\}$, where $K_{4}^{3-}$ is the $3$-graph on $4$ vertices with exactly $3$ edges, and $F_5 \coloneqq \{123,124,345\}$. 
Here is another example related to hypergraph expansions. 

Suppose that $\ell \ge r \ge 2$ are integers.
An $r$-graph $F$ is a \textbf{weak expansion} of $K_{\ell+1}$ if it can be obtained from $K_{\ell+1}$ by adding $r-2$ vertices into each edge. 
A key difference from the expansion $H_{\ell+1}^{r}$ of $K_{\ell+1}$ is that these added $(r-2)$-sets do not need to be pairwise disjoint. 
We use $\mathcal{K}_{\ell+1}^{r}$ to denote the collection of all weak expansions of $K_{\ell+1}$.

\begin{fact}\label{FACT:weak-expansion}
    Let $\ell \ge r \ge 2$ be integers.
    The family $\mathcal{K}_{\ell+1}^{r}$ is blowup-invariant and $\mathcal{K}_{\ell+1}^{r} \le_{\mathrm{hom}} \{H_{F}^{r}\}$ for every graph $F$ with $\chi(F) = \ell+1$.  
    In addition, every symmetrized $\mathcal{K}_{\ell+1}^{r}$-free $r$-graph is $\ell$-partite\footnote{An $\ell'$-partite $r$-graph is automatically $\ell$-partite for every $\ell \ge \ell'$ since we allow empty parts.}.
\end{fact}

The motivation for the definition of $\le_{\mathrm{hom}}$ comes from the following fact, which follows easily from the Removal Lemma~\cite{RS09}. 

\begin{fact}\label{FACT:le-hom}
    Suppose that $\mathcal{F}, \hat{\mathcal{F}}$ are two families of $r$-graphs satisfying $\hat{\mathcal{F}} \le_{\mathrm{hom}} \mathcal{F}$. 
    Then every $n$-vertex $\mathcal{F}$-free $r$-graph can be made $\hat{\mathcal{F}}$-free by removing $o(n^r)$ edges. 
    In particular, $\pi(Q, \mathcal{F}) \le \pi(Q, \hat{\mathcal{F}})$ for every $r$-graph $Q$. 
\end{fact}

For simplicity of use, we present the following corollaries of Theorems~\ref{THM:GeneralFunction-Stability},~\ref{THM:GeneralFunction-Stability-b}, and~\ref{THM:GeneralFunction-Stability-c} within the context of generalized Tur\'{a}n problems. 
These corollaries are obtained by considering the map $\Gamma \colon \mathfrak{G}^{r} \to \mathbb{R}$ defined by 
\begin{align*}
    \Gamma(\mathcal{H})
    \coloneqq \mathrm{inj}(Q,\mathcal{H}) \cdot \mathbbm{1}_{\mathcal{F}}(\mathcal{H}). 
\end{align*}

Let $r \ge 2$ be an integer, $Q$ be an $r$-graph, and $\mathcal{F}$ be a family of $r$-graphs. 
Let $\mathfrak{H}$ be a hereditary family of $\mathcal{F}$-free $r$-graphs. 
\begin{enumerate}[label=(\roman*)]
    \item We say $Q$ is \textbf{symmetrization-increasing} if the map $\Gamma \colon \mathfrak{G}^{r} \to \mathbb{R}$ defined by $\Gamma(\mathcal{H}) \coloneqq \mathrm{inj}(Q, \mathcal{H})$ is symmetrization-increasing.
    \item We say $\mathcal{F}$ is \textbf{symmetrized-stable} with repsect to $\mathfrak{H}$ if every symmetrized $\mathcal{F}$-free $r$-graph is contained in $\mathcal{H}$. 
    \item We say $\mathcal{F}$ is \textbf{$Q$-edge-stable} with respect to $\mathfrak{H}$ if for every $\varepsilon>0$ there exist $\delta>0$ and $N_0$ such that every $\mathcal{F}$-free $r$-graph $\mathcal{H}$ on $n \ge N_0$ vertices with $\mathrm{inj}(Q, \mathcal{H}) \ge (1-\delta) \cdot \mathrm{inj}(n,Q,\mathcal{F})$ is contained in $\mathfrak{H}$ after removing at most $\varepsilon n^r$ edges. 
    \item We say $\mathcal{F}$ is \textbf{$Q$-degree-stable} with respect to $\mathfrak{H}$ if there exist $\varepsilon>0$ and $N_0$ such that every $\mathcal{F}$-free $r$-graph $\mathcal{H}$ on $n \ge N_0$ vertices with $d_{Q}(\mathcal{H}) \ge (1-\varepsilon) \cdot \frac{v(Q)\cdot \mathrm{inj}(n,Q,\mathcal{F})}{n}$ is contained in $\mathfrak{H}$. 
    \item We say $\mathcal{F}$ is \textbf{$Q$-vertex-extendable} with respect to $\mathfrak{H}$ if there exist $\varepsilon>0$ and $N_0$ such that the following holds for every $\mathcal{F}$-free $r$-graph $\mathcal{H}$ on $n \ge N_0$ vertices with $d_{Q}(\mathcal{H}) \ge (1-\varepsilon) \cdot \frac{v(Q)\cdot \mathrm{inj}(n,Q,\mathcal{F})}{n} \colon$ 
    if $\mathcal{H}-v \in \mathfrak{H}$ for some $v \in V(\mathcal{H})$, then $\mathcal{H} \in \mathfrak{H}$. 
\end{enumerate}

The following theorem is a consequence of Facts~\ref{FACT:blowup-invariant-symmetrization-increasing},~\ref{FACT:inj-basic-properties}, Theorems~\ref{THM:GeneralFunction-Stability}, and~\ref{THM:GeneralFunction-Stability-c}. 

\begin{theorem}\label{THM:general-generalized-Turan-a}
    Let $r \ge 2$ be an integer and $Q$ be an $r$-graph. 
    Let $\mathcal{F}, \hat{\mathcal{F}}$ be two families of $r$-graphs such that $\hat{\mathcal{F}} \le_{\mathrm{hom}} \mathcal{F}$ and $\pi_{\mathrm{inj}}(Q, \mathcal{F}) > 0$.  
    Let $\mathfrak{H}$ is a hereditary family of $(\hat{\mathcal{F}} \cup \mathcal{F} )$-free $r$-graphs. 
    Suppose that  
    \begin{enumerate}[label=(\roman*)] 
        \item\label{THM:general-generalized-Turan-a-1} $Q$ is symmetrization-increasing, 
        \item\label{THM:general-generalized-Turan-a-2} $\hat{\mathcal{F}}$ is blowup-invariant, 
        \item\label{THM:general-generalized-Turan-a-3} $\hat{\mathcal{F}}$ is symmetrized-stable with respect to $\mathfrak{H}$, and
        \item\label{THM:general-generalized-Turan-a-4} both $\hat{\mathcal{F}}$ and $\mathcal{F}$ are $Q$-vertex-extendable with respect to $\mathfrak{H}$. 
    \end{enumerate}
    Then $\mathcal{F}$ is $Q$-degree-stable with respect to $\mathfrak{H}$. 
\end{theorem}
\begin{proof}[Proof of Theorem~\ref{THM:general-generalized-Turan-a}]
    Let the maps $\hat{\Gamma}, \Gamma \colon \mathfrak{G}^{r} \to \mathbb{R}$ be defined by 
    \begin{align*}
        \hat{\Gamma}(\mathcal{H})
        \coloneqq \mathrm{inj}(Q,\mathcal{H}) \cdot \mathbbm{1}_{\hat{\mathcal{F}}}(\mathcal{H})
        \quad\text{and}\quad 
        \Gamma(\mathcal{H})
        \coloneqq \mathrm{inj}(Q,\mathcal{H}) \cdot \mathbbm{1}_{\mathcal{F}}(\mathcal{H}). 
    \end{align*}
    Note that $\mathrm{ex}_{\Gamma}(n) = \mathrm{inj}(n,Q,\mathcal{F})$ and $\mathrm{ex}_{\hat{\Gamma}}(n) = \mathrm{inj}(n,Q,\hat{\mathcal{F}})$. 
    
    Due to Assumptions~\ref{THM:general-generalized-Turan-a-2},~\ref{THM:general-generalized-Turan-a-3},~\ref{THM:general-generalized-Turan-a-4}, Facts~\ref{FACT:blowup-invariant-symmetrization-increasing},~\ref{FACT:inj-basic-properties}, and Theorem~\ref{THM:GeneralFunction-Stability}, we know that $\hat{F}$ is degree-stable with respect to $\mathfrak{H}$. In particular, $\hat{F}$ is edge-stable with respect to $\mathfrak{H}$. 
    By Theorem~\ref{THM:GeneralFunction-Stability-c}, it suffices to show that $\hat{\Gamma}$ is a trajectory of $\Gamma$. 
    
    Fix $n \in \mathbb{N}$. 
    By Fact~\ref{FACT:exact-symm-increase}, there exists an $n$-vertex $r$-graph $\mathcal{H} \in \mathfrak{H}$ such that $\hat{\Gamma}(\mathcal{H}) =  \mathrm{inj}(n,Q,\hat{\mathcal{F}})$. 
    By the assumption on $\mathfrak{H}$, the $r$-graph $\mathcal{H}$ is also $\mathcal{F}$-free. 
    It follows that $\mathrm{inj}(n,Q,\mathcal{F}) \ge \Gamma(\mathcal{H}) = \hat{\Gamma}(\mathcal{H}) =  \mathrm{inj}(n,Q,\hat{\mathcal{F}})$. 
    On the other hand, it follows from Fact~\ref{FACT:le-hom} that every $\mathcal{F}$-free $r$-graph on $n$ vertices can be made $\hat{F}$-free by removing $o(n^r)$ edges. 
    This implies that there exists $\mathcal{H}' \subset \mathcal{H}$ such that $|\mathcal{H}'| = |\mathcal{H}| - o(n^r)$ and $\hat{\Gamma}(\mathcal{H}') \ge \Gamma(\mathcal{H}) - o(\mathrm{inj}(n,Q,\mathcal{F}))$. 
    Therefore, $\hat{\Gamma}$ is a trajectory of $\Gamma$. 
\end{proof}

The following theorem is a direct corollary of Fact~\ref{FACT:inj-basic-properties} and Theorem~\ref{THM:GeneralFunction-Stability-b}. 

\begin{theorem}\label{THM:general-generalized-Turan-b}
    Let $r \ge 2$ be an integer, $Q$ be an $r$-graph, and $\mathcal{F}$ be a family of $r$-graphs with $\pi_{\mathrm{inj}}(Q, \mathcal{F}) > 0$.  
    Let $\mathfrak{H}$ is a hereditary family of $\mathcal{F}$-free $r$-graphs. 
    Suppose that $\mathcal{F}$ is $Q$-degree-stable and $Q$-vertex-extendable with respect to $\mathfrak{H}$. 
    Then $\mathcal{F}$ is $Q$-degree-stable with respect to $\mathfrak{H}$. 
\end{theorem}

The following proposition shows that in the proofs of Theorems~\ref{THM:Homomorphism-complete-multipartite-hypergraph},~\ref{THM:Homomorphism-complete-multipartite-hypergraph-b}, and~\ref{THM:Turan-good-edge-critical-expansion}, it suffices to establish only the degree-stability part. 

\begin{proposition}\label{PROP:min-degree-extremal}
    Let $r \ge 2$ be an integer, $Q$ be an $r$-graph, and $F$ be an edge-critical graph with $\chi(F) = \ell+1$. 
    Suppose that $\pi_{\mathrm{inj}}(Q,H_{F}^{r}) > 0$ and $H_{F}^{r}$ is $Q$-degree-stable with respect to $\mathfrak{K}_{\ell}^{r}$. 
    Then for large $n$, 
    \begin{align*}
        \mathrm{inj}(n,Q,H_{F}^{r})
        = \max\left\{\mathrm{inj}(Q,\mathcal{H}) \colon \mathcal{H} \in \mathfrak{K}_{\ell}^{r} \text{ and } v(\mathcal{H}) = n\right\}. 
    \end{align*}
\end{proposition}
\begin{proof}
    Let $0 < \varepsilon \ll \delta$ be sufficiently small and $n$ be sufficiently large. 
    Let $\mathcal{H}$ be an $n$-vertex $H_{F}^{r}$-free $r$-graph with $\mathrm{inj}(Q,\mathcal{H}) =  \mathrm{inj}(n,Q,H_{F}^{r})$. 
    We aim to show that $\mathcal{H} \in \mathfrak{K}_{\ell}^{r}$, which by the $Q$-degree-stability of $H_{F}^{r}$, is reduced to show that $\delta_{Q}(\mathcal{H}) \ge \frac{v(Q) \cdot \mathrm{inj}(n,Q,H_{F}^{r})}{n} -\delta n^{v(Q)}$. 
    Let 
    \begin{align*}
        Z_{\varepsilon}
        \coloneqq \left\{v\in V(\mathcal{H}) \colon d_{Q,\mathcal{H}}(v) \le (1-\varepsilon) \frac{v(Q) \cdot \mathrm{inj}(n,Q,H_{F}^{r})}{n}\right\}. 
    \end{align*}
    Then similar to Lemma~\ref{LEMMA:Z-delta-size}, we have $|Z_{\varepsilon}| \le \delta n/4$, and hence, $\mathcal{H}_1 \coloneqq \mathcal{H}-Z_{\varepsilon}$ satisfies 
    \begin{align*}
        \delta_{Q}(\mathcal{H}_1)
        \ge (1-\varepsilon) \frac{v(Q) \cdot \mathrm{inj}(n,Q,H_{F}^{r})}{n} - |Z_{\varepsilon}| n^{v(Q)-2}
        \ge \frac{v(Q) \cdot \mathrm{inj}(n,Q,H_{F}^{r})}{n} -\frac{\delta n^{v(Q)}}{2}. 
    \end{align*}
    It follows from the $Q$-degree-stability of $H_{F}^{r}$ that $\mathcal{H}_1 \in \mathfrak{K}_{\ell}^{r}$. 
    Suppose to the contrary that there exists a vertex  $v\in V(\mathcal{H})$ such that $d_{Q,\mathcal{H}}(v) \le \frac{v(Q) \cdot \mathrm{inj}(n,Q,H_{F}^{r})}{n} -\delta n^{v(Q)}$ (note that $v$ must be contained in $Z_{\varepsilon}$). 
    Let $u \in V(\mathcal{H}_1)$ be a vertex of maximum $Q$-degree (in $\mathcal{H}_1$). In particular, $d_{Q,\mathcal{H}_1}(u) \ge \frac{v(Q) \cdot \mathrm{inj}(n,Q,H_{F}^{r})}{n} -\frac{\delta n^{v(Q)}}{2} > d_{Q,\mathcal{H}}(v)$. 
    Consider the following $r$-graph 
    \begin{align*}
        \hat{\mathcal{H}}
        \coloneqq \left(\mathcal{H} - v\right)
            \cup \left\{e\cup \{v\} \colon e\in L_{\mathcal{H}_1}(u)\right\}. 
    \end{align*}
    Notice that $\hat{\mathcal{H}}$ is still $H_{F}^{r}$-free, but $|\hat{\mathcal{H}}|
        \ge |\mathcal{H}| - d_{Q,\mathcal{H}}(v) + d_{Q,\mathcal{H}_1}(u) > \mathrm{inj}(n,Q,H_{F}^{r})$, a contradiction. 
\end{proof}
\section{Proof of Theorem~\ref{THM:AES-pentagon}}\label{SEC:Proof-Pentagon}
We present the short proof for Theorem~\ref{THM:AES-pentagon} in this section. 
Let $\mathfrak{C}_{5}$ be the collection of all $C_5$-colorable graphs. 
Recall from Theorem~\ref{THM:Pentagon-stability} that $K_3$ is $C_5$-edge-stable with respect to $\mathfrak{C}_{5}$. 
Therefore, by Theorem~\ref{THM:general-generalized-Turan-b}, it suffices to show that $K_3$ is $C_5$-vertex-extendable with respect to $\mathfrak{C}_{5}$, which is equivalent to the following proposition.

\begin{proposition}\label{PROP:Pentagon-vtx-ext}
    There exist $\varepsilon>0$ and $N_0$ such that the following holds for all $n \ge N_0$. 
    Suppose that $G$ is an $n$-vertex triangle-free graph with $\delta_{C_5}(G) \ge (10-\varepsilon) \left(n/5\right)^4$, and there exists a vertex $v\in V(G)$ such that $G-v \in \mathfrak{C}_5$. 
    Then $G \in \mathfrak{C}_5$ as well. 
\end{proposition}
\begin{proof}
    Let $V \coloneqq V(G)$, $U \coloneqq V\setminus \{v\}$, and $G' \coloneqq G - v$. 
    Let $N(C_5, G) \coloneqq \frac{\mathrm{inj}(C_5,G)}{|\mathrm{Aut}(C_5)|} = \frac{\mathrm{inj}(C_5,G)}{10}$ denote the number of (unlabelled) copies of $C_5$ in $G$. 
    For simplicity in indices, we equate $C_5$ with $\mathbb{Z}_{5}$, where the edge set is $\{\{1,2\},\{2,3\},\{3,4\},\{4,5\},\{5,1\}\}$.   
    Let $\psi\colon G' \to \mathbb{Z}_{5}$ be a homomorphism and $V_i \coloneqq \psi^{-1}(i)$ for $i\in \mathbb{Z}_{5}$. 
    Additionally, let $x_i \coloneqq |V_i|/n$ for $i\in \mathbb{Z}_{5}$. 
    It follows from the minimum $C_5$-degree assumption that 
    \begin{align*}
        \prod_{i=1}^{5}(x_i n) 
        \ge N(C_5, G')
        \ge \frac{n-1}{5}\left((1-\varepsilon) \left(\frac{n}{5}\right)^4 - n^3\right)
        \ge (1-2\varepsilon)\left(\frac{n}{5}\right)^5. 
    \end{align*}
    Simple calculations (see e.g.~{\cite[Lemma~2.2]{LMR23unif}}) show that 
    \begin{align}\label{equ:xi-size}
        \left|x_i - 1/5\right| 
        \le 2\sqrt{\varepsilon}
        \quad\text{for all}\quad i\in \mathbb{Z}_5. 
    \end{align}
    Consequently, for every $i \in \mathbb{Z}_{5}$ and for every $u\in V_i$ we have  
    \begin{align}\label{equ:G'-min-deg}
        \min\{|N_{G'}(u) \cap V_{i-1}|, |N_{G'}(u) \cap V_{i+1}|\}
        > \left(1/5 - 8\sqrt{\varepsilon}\right)n, 
    \end{align}
    since otherwise, by~\eqref{equ:xi-size}, the number of (unlabelled) $C_5$ in $G$ containing $v$ satisfies 
    \begin{align*}
        \frac{d_{C_5, G'}(u)}{|\mathrm{Aut}(C_5)|}
        \le \left(\frac{1}{5}+2\sqrt{\varepsilon}\right)^3 \left(\frac{1}{5} - 8\sqrt{\varepsilon}\right)n^4
        < (1-2\varepsilon)\left(\frac{n}{5}\right)^4
        < (1-\varepsilon) \left(\frac{n}{5}\right)^4 - n^3,  
    \end{align*}
    contradicting the minimum $C_5$-degree assumption. 
    Let $N_i \coloneqq V_i \cap N_{G}(v)$ for $i\in \mathbb{Z}_{5}$. 
    Observe from~\eqref{equ:xi-size} and~\eqref{equ:G'-min-deg} that if $N_i \neq \emptyset$ for some $i\in \mathbb{Z}_{5}$, then 
    \begin{align}\label{equ:Ni-1-size}
        \max\{|N_{i-1}|, |N_{i+1}|\}
        < 10 \sqrt{\varepsilon} n. 
    \end{align}
    Indeed, if, for instance, $|N_{i-1}| \ge 10 \sqrt{\varepsilon} n$, then any vertex $u_1 \in N_i$ would have a neighbor in $N_{i-1}$, which together with $v$ form a triangle, a contradiction. 
    \begin{claim}\label{CLAIM:unique-i}
        There exists a unique $i_{\ast} \in \mathbb{Z}_{5}$ such that $\min\{|N_{i_{\ast}}|, |N_{i_{\ast}+2}|\} \ge 10\sqrt{\varepsilon} n$. 
    \end{claim}
    \begin{proof}
        First, it follows from the assumption $d_{C_5, G}(v)/|\mathrm{Aut}(C_5)| \ge (1-\varepsilon) \left(n/5\right)^4$ and simple calculations that there are at least two sets in $\{N_1, \ldots, N_5\}$ of size at least $10\sqrt{\varepsilon}n$. 
        On the other hand,~\eqref{equ:Ni-1-size} shows that every pair $N_i, N_j$ of size at least at least $10\sqrt{\varepsilon}n$ must satisfy $|i-j|>1$. 
        Now, this claim follows easily from the structure of $C_5$. 
    \end{proof}
    It follows from Claim~\ref{CLAIM:unique-i} and~\eqref{equ:Ni-1-size} that $N_j = \emptyset$ for all $j \in \mathbb{Z}_{5}\setminus \{i_{\ast}, i_{\ast+2}\}$. 
    Therefore, $G$ is $C_5$-colorable.
\end{proof}

\section{Vertex-extendability for general hypergraphs}
The main purpose of this section is to establish the following result, which will be crucial in establishing the $Q$-vertex-extendability later.  
Recall the definition of $Q$-Lagrangian from Section~\ref{SEC:Prelim}. 

\begin{proposition}\label{PROP:Q-vtx-ext-F-expansion}
    Let $\ell \ge k \ge r \ge 2$ be integers, $Q$ be an $r$-graph with $\chi(Q) = k$, and $F$ be an edge-critical graph with $\chi(F) = \ell+1$. 
    There exist $\varepsilon >0$ and $N_0$ such that the following holds for all $n \ge N_0$. 
    Suppose that $\mathcal{H}$ is an $n$-vertex $H_{F}^{r}$-free $r$-graph with $\delta_{Q}(\mathcal{H}) \ge \left(\lambda_{Q}(K_{\ell}^{r})-\varepsilon \right) v(Q) n^{v(Q)-1}$ and $v\in V(\mathcal{H})$ is a vertex such that $\mathcal{H} - v \in \mathfrak{K}_{\ell}^{r}$. 
    Then $\mathcal{H} \in \mathfrak{K}_{\ell}^{r}$ as well. 
    In particular, if $\pi_{\mathrm{inj}}(Q,H_{F}^{r}) = \lambda_{Q}(K_{\ell}^{r})$, then $H_{F}^{r}$ is $Q$-vertex-extendable with respect to $\mathfrak{K}_{\ell}^{r}$.
\end{proposition}

The following fact follows from the theory of Lagrange Multiplier and Compactness. 
For convenience, for every vector $\vec{x} \coloneqq (x_1, \ldots, x_n) \in \mathbb{R}^{n}$ we let $\mathrm{Supp}(\vec{x}) \coloneqq \left\{i\in [n] \colon x_i > 0\right\}$.  

\begin{lemma}\label{LEMMA:Lagrangian-Multiplier}
    Let $Q$ be an $r$-graph and $\mathcal{H}$ be an $r$-graph on $[n]$. 
    \begin{enumerate}[label=(\roman*)]
        \item\label{LEMMA:Lagrangian-Multiplier-1} 
        If $\vec{x} \coloneqq (x_1, \ldots, x_{n}) \in \mathrm{Opt}(Q,\mathcal{H})$, then 
        \begin{align*}
            D_{i} P_{Q,\mathcal{H}}(x_1, \ldots, x_{n}) = v(Q) \cdot \lambda_{Q}(\mathcal{H})
            \quad\text{for every}\quad i\in \mathrm{Supp}(\vec{x}). 
        \end{align*} 
        \item\label{LEMMA:Lagrangian-Multiplier-2} 
        For every $\delta>0$ there exists $\varepsilon>0$ such that if $(y_1, \ldots, y_{n}) \in \Delta^{n-1}$ satisfies $P_{Q,\mathcal{H}}(y_1, \ldots, y_{n}) \ge \lambda_{Q}(\mathcal{H}) - \varepsilon$, then there exists $\vec{x} \coloneqq (x_1, \ldots, x_n) \in \mathrm{Opt}(Q,\mathcal{H})$ such that $\max_{i\in [n]}|x_i - y_i| \le \delta$ and
        \begin{align*}
            \left| D_{i} P_{Q,\mathcal{H}}(y_1, \ldots, y_{n}) - D_{i} P_{Q,\mathcal{H}}(x_1, \ldots, x_{n}) \right| 
            \le \delta
            \quad\text{for every}\quad i\in \mathrm{Supp}(\vec{x}). 
        \end{align*}
    \end{enumerate}
\end{lemma}
\begin{proof}[Proof of Lemma~\ref{LEMMA:Lagrangian-Multiplier}]
    We prove~\ref{LEMMA:Lagrangian-Multiplier-1} first. 
    By relabeling the vertices of $\mathcal{H}$ we may assume that $x_i > 0$ iff $i\in [m]$ for some $m \le n$. 
    Let $\mathcal{H}'$ denote the induced subgraph of $\mathcal{H}$ on $[m]$. Then $(x_1, \ldots, x_{m})$ is an vector in $\Delta^{m-1}$ and $P_{Q,\mathcal{H}'}(x_1, \ldots, x_{m}) = \lambda_{Q}(\mathcal{H}) = \lambda_{Q}(\mathcal{H}')$. 
    It follows from the theory of Lagrangian Multiplier that there exists a constant $\mu$ such that 
    \begin{align*}
        D_{i} P_{Q,\mathcal{H}'}(x_1, \ldots, x_{m}) = \mu
        \quad\text{for all}\quad i \in [m]. 
    \end{align*}
    Since $P_{Q,\mathcal{H}'}$ is a homogeneous polynomial of degree $v(Q)$, we have 
    \begin{align*}
        \mu 
        =\sum_{i\in [m]}x_i \mu 
        = \sum_{i\in [m]} x_i \cdot D_{i} P_{Q,\mathcal{H}'}(x_1, \ldots, x_{m}) 
        = v(Q) \cdot P_{Q,\mathcal{H}'}(x_1, \ldots, x_{m})
        = v(Q) \cdot \lambda_{Q}(\mathcal{H}). 
    \end{align*}
    This proves~\ref{LEMMA:Lagrangian-Multiplier-1}. 

    Now we consider~\ref{LEMMA:Lagrangian-Multiplier-2}. 
    For every $\alpha \ge 0$ define 
    \begin{align*}
        \mathrm{Opt}_{\alpha}(Q, \mathcal{H}) 
         \coloneqq \left\{(x_1, \ldots, x_n) \in \Delta^{n-1} \colon \exists (x'_1, \ldots, x'_n) \in \mathrm{Opt}(Q, \mathcal{H}) \text{ with } \max_{i\in [n]}|x_i-x_i'| \le \alpha\right\}
    \end{align*}
    and 
    \begin{align*}
        \lambda_{Q,\alpha}(\mathcal{H})
         \coloneqq \sup\left\{P_{Q,\mathcal{H}}(x_1, \ldots, x_{n}) \colon (x_1, \ldots, x_{n}) \in \Delta^{n-1} \setminus \mathrm{Opt}_{\alpha}(Q, \mathcal{H})\right\}. 
    \end{align*}
    It is clear that $\lambda_{Q,\alpha}(\mathcal{H}) < \lambda_{Q}(\mathcal{H})$ for every $\alpha>0$, and $\lambda_{Q,\alpha}(\mathcal{H})$ is decreasing in $\alpha$. 

    Let $C_{Q} \coloneqq \max_{(z_1, \ldots, z_n) \in \Delta^{n-1}}\max_{i\in [n]}\left|D_{i} P_{Q,\mathcal{H}}(z_1, \ldots, z_{n})\right|$. 
    Fix $\delta>0$ and let 
    \begin{align*}
        \varepsilon 
        \coloneqq \min\left\{\frac{\lambda_{Q}(\mathcal{H}) - \lambda_{Q,\delta}(\mathcal{H})}{2}, \frac{\delta}{C_{Q}}\right\}. 
    \end{align*}
    Suppose that $(y_1, \ldots, y_{n}) \in \Delta^{n-1}$ satisfies $P_{Q,\mathcal{H}}(y_1, \ldots, y_{n}) \ge \lambda_{Q}(\mathcal{H}) - \varepsilon > \lambda_{Q,\delta}(\mathcal{H})$. 
    Then it follows from the definition that $\max_{i\in [n]}|y_i - x_i| \le \varepsilon \le \delta$ for some $(x_1, \ldots, x_n) \in \mathrm{Opt}(Q, \mathcal{H})$. 
    By Taylor's remainder theorem, for every $i \in [n]$ satisfying $x_i > 0$ we have 
    \begin{align*}
         \left|D_{i} P_{Q,\mathcal{H}}(y_1, \ldots, y_{n}) - D_{i} P_{Q,\mathcal{H}}(x_1, \ldots, x_{n}) \right|
         \le C_{Q} \cdot \max_{i\in [n]}|y_i -x_i|
         \le \delta,  
    \end{align*}
    completing the proof of~\ref{LEMMA:Lagrangian-Multiplier-2}. 
\end{proof}

\begin{lemma}\label{LEMMA:opt-solotions-clique-expansion-general}
     Let $\ell \ge k \ge r\ge 2$ be integers and $Q$ be an $r$-graph with $\chi(Q) = k$.
     There exist $\varepsilon>0$ and $N_0>0$ such that the following holds for all $n \ge N_0$. 
     Suppose that $V_1 \cup \cdots \cup V_{\ell} = [n]$ is a partition such that $\mathrm{inj}(Q, K^{r}(V_1, \ldots, V_{\ell}))$ is maximized. 
     Then $|V_i| \ge \varepsilon n$ for all $i\in [\ell]$. 
     In particular, there exists a constant $\varepsilon_{Q,\ell} > 0$ depending only on $Q$ and $\ell$ such that $\min_{i\in [\ell]} x_i \ge \varepsilon_{Q,\ell}$ for every $(x_1, \ldots, x_{\ell}) \in \mathrm{Opt}(Q,K_{\ell}^{r})$. 
\end{lemma}
\begin{proof}[Proof of Lemma~\ref{LEMMA:opt-solotions-clique-expansion-general}]
    Due to Fact~\ref{FACT:homomorphism-shadow}, it suffices to prove this lemma for the case $r=2$.
    Fix a graph $Q$ with $\chi(Q) = k$.
    Let $0 < \varepsilon \ll \hat{\varepsilon} \ll 1$ be sufficiently small and $N_0$ be sufficiently large.
    Suppose that this lemma is not true and let $n \ge N_0$ be an integer such that there exists a partition $V_1 \cup \cdots \cup V_{\ell} = [n]$ with $\mathrm{inj}(Q, K^{r}(V_1, \ldots, V_{\ell}))$ being maximized, but $\min_{i\in [\ell]}|V_i| \le \varepsilon n$. 
    By relabeling the sets we may assume that $V_1$ has the smallest size and $V_2$ has the largest size. 
    Then $|V_1|\le \varepsilon n$, and by the Pigeonhole Principle, we have $|V_2|\ge n/\ell$.
    For simplicity, let $K \coloneqq K^{2}[V_1, \ldots, V_{\ell}]$. 
    \begin{claim}\label{CLAIM:opt-solotions-clique-expansion-general-a}
        There exists a $k$-set $J \subset [\ell]$ such that $\min_{j\in J}|V_j| \ge \hat{\varepsilon} n$.
    \end{claim}
    \begin{proof}[Proof of Claim~\ref{CLAIM:opt-solotions-clique-expansion-general-a}]
        Suppose to the contrary that for every $k$-set $J \subset [\ell]$ there exists $i \in J$ such that $|V_i| < \hat{\varepsilon} n$. 
        Then we would have 
        \begin{align*}
            \mathrm{inj}(Q,K)
            \le \sum_{i \ge k}\binom{\ell}{i} \cdot \hat{\varepsilon}n \cdot  n^{v(Q)-1}
            < \binom{\ell}{k}\left(\left\lfloor\frac{n}{\ell}\right\rfloor-v(Q)\right)^{v(Q)}
            \le \mathrm{inj}(Q,T(n,\ell)), 
        \end{align*}
        contradicting the maximality of $K$. 
    \end{proof}
    By symmetry, we may assume that $J = \{2, \ldots, k+1\}$ is the $k$-set guaranteed by Claim~\ref{CLAIM:opt-solotions-clique-expansion-general-a}. 
    Let $U_1 \cup U_2 = V_2$ be a balanced partition.  Let $G$ be the graph obtained from $K$ by removing $V_1$ and adding all edges crossing $U_1$ and $U_2$. Observe that $G$ is a complete $\ell$-partite graph with parts $U_1, U_2, V_3, \ldots, V_{\ell}$. 
    Suppose that $W_1 \cup \cdots \cup W_{k} = V(Q)$ is a partition such that $Q \subset K^2[W_1, \ldots, W_k]$. Note that this is possible and $\min_{1\le i < j \le k}|G[W_i,W_j]| \ge 1$ since $\chi(Q) = k$.  
    Observe that every homomorphism from $Q$ to $G$ that maps $W_1$ to $U_1$ and $W_2$ to $U_2$ is not contained in $\mathrm{Inj}(Q,K)$. 
    Therefore, 
    \begin{align*}
        & \quad \mathrm{inj}(Q, G) - \mathrm{inj}(Q, K) \\
        & \ge \left(|U_1|-v(Q)\right)^{|W_1|}
            \left(|U_2|-v(Q)\right)^{|W_2|}
            \prod_{i=3}^{k}\left(|V_i|-v(Q)\right)^{|W_i|}
            - v(Q) |V_1| n^{v(Q)-1} \\
        & \ge \left(\frac{1}{4\ell}\right)^{|W_1|+|W_2|} \cdot \left(\frac{\hat{\varepsilon}}{2}\right)^{|W_3|+\cdots +|W_{\ell}|} \cdot n^{v(Q)}
            - \varepsilon \cdot v(Q) \cdot  n^{v(Q)}
            > 0, 
    \end{align*}
    contradicting the maximality of $K$. 
\end{proof}

\begin{lemma}\label{LEMMA:homo-subgraph-complete-partite}
    Let $\ell \ge r \ge 2$ be integers and $Q$ be a non-empty $r$-graph. 
    Let $V_1 \cup \cdots \cup V_{\ell} = [n]$ be a partition and $x_{\ast} \coloneqq \min\{|V_i|/n \colon i\in [\ell]\}$. 
    Suppose that $\mathcal{H}$ is a spanning subgraph of the complete $\ell$-partite $r$-graph $K^r \coloneqq K^r[V_1, \ldots, V_{\ell}]$ and $n \ge 2v(Q)/x_{\ast}$. 
    Then 
    \begin{enumerate}[label=(\roman*)]
        \item\label{LEMMA:homo-subgraph-complete-partite-1} $\mathrm{inj}(Q, K^r) - \mathrm{inj}(Q, \mathcal{H}) \ge |K^r \setminus \mathcal{H}|\left(\frac{x_{\ast}n}{2}\right)^{v(Q)-r}$, and 
        \item\label{LEMMA:homo-subgraph-complete-partite-2} $d_{Q, K^r}(v) - d_{Q, \mathcal{H}}(v) \ge \left(d_{K^r}(v) - d_{\mathcal{H}}(v)\right) \left(\frac{x_{\ast}n}{2}\right)^{v(Q)-r}$ for every $v\in [n]$. 
    \end{enumerate}
\end{lemma}
\begin{proof}
    Let $\overline{\mathcal{H}} \coloneqq K^r \setminus \mathcal{H}$. 
    Fix an edge $e_{\star}\in Q$ and fix an edge $e_{\ast} \in \overline{\mathcal{H}}$. 
    Notice that the set 
    \begin{align*}
        \mathcal{M}(e_{\ast}) 
        \coloneqq 
        \left\{\psi \in \mathrm{Inj}(Q, K^r) \colon \psi(e_{\star}) = e_{\ast}\right\} 
    \end{align*}
    has size at least $\left(x_{\ast}n - v(Q)\right)^{v(Q)-r} \ge \left(x_{\ast}n/2\right)^{v(Q)-r}$ (since every vertex in $V(Q) \setminus e_{\star}$ has at least $x_{\ast}n - v(Q)$ choices). 
    Since every map in $\mathcal{M}(e_{\ast})$ is not contained in $\mathrm{Inj}(Q, \mathcal{H})$, we have 
    \begin{align*}
        \mathrm{inj}(Q, K^r) - \mathrm{inj}(Q, \mathcal{H})
        \ge \sum_{e_{\ast} \in \overline{\mathcal{H}}} |\mathcal{M}(e_{\ast})|
        \ge |\overline{\mathcal{H}}|\left(\frac{x_{\ast}n}{2}\right)^{v(Q)-r}, 
    \end{align*}
    proving the first part. 

    Similarly, for every $v\in [n]$ we have 
    \begin{align*}
        d_{Q, K^r}(v) - d_{Q, \mathcal{H}}(v)
        \ge \sum_{e_{\ast} \in \overline{\mathcal{H}} \colon v\in e_{\ast}} |\mathcal{M}(e_{\ast})|
        \ge d_{\overline{\mathcal{H}}}(v)\left(\frac{x_{\ast}n}{2}\right)^{v(Q)-r},  
    \end{align*}
    proving the second part. 
\end{proof}

The following lemma is an improvement of~{\cite[Lemma~4.8]{LMR23unif}}. 

\begin{lemma}\label{LEMMA:vertex-extendability-clique-expansion}
    Let $\ell \ge r \ge 2$ be integers and $F$ be an edge-critical graph with $\chi(F) = \ell+1$. 
    For every $\varepsilon>0$ and $\hat{\varepsilon}>0$ there exist $\delta>0$ and $N_{0}>0$ such that the following holds for all $n \ge N_{0}$. 
    Let $V_1 \cup \cdots \cup V_{\ell} = [n]$ be a partition with $|V_i| \ge \varepsilon n$ for $i\in [\ell]$. 
    Suppose that $\mathcal{H}$ is an $(n+1)$-vertex $H_{F}^{r}$-free $r$-graph with a vertex $v_{\ast}\in V(\mathcal{H})$ such that 
    \begin{enumerate}[label=(\roman*)]
        \item $\mathcal{H}' \coloneqq \mathcal{H}-v_{\ast}$ is a spanning subgraph of $K^r \coloneqq K^{r}(V_1, \ldots, V_{\ell})$, and 
        \item $d_{\mathcal{H}'}(u) \ge d_{K}(u) - \delta n^{r-1}$ for all $u \in V(\mathcal{H}')$.
    \end{enumerate}
    Then the following statements hold. 
    \begin{enumerate}[label=(\roman*)]
        \item Every $e\in L_{\mathcal{H}}(v_{\ast})$ satisfies $|e\cap V_i| \le 1$ for $i\in [\ell]$. 
        \item 
        Let $C_{F}\coloneqq v(F)+(r-2)|F|$. Then sets
        \begin{align*}
            I
             \coloneqq 
            \left\{i\in [\ell] \colon N_{\mathcal{H}}(v_{\ast}) \cap V_i \neq \emptyset \right\} \quad\text{and}\quad
            I_{\mathrm{large}}
             \coloneqq 
            \left\{i\in [\ell] \colon |N_{\mathcal{H}}^{C_{F}}(v_{\ast}) \cap V_i| \ge \hat{\varepsilon} n\right\}
        \end{align*}
        satisfy either $|I_{\mathrm{large}}| \le \ell-2$ or $|I| = |I_{\mathrm{large}}| = \ell-1$.
        In particular, in the latter case, $\mathcal{H}$ is $\ell$-partite.
    \end{enumerate}
\end{lemma}

In the proof of Lemma~\ref{LEMMA:vertex-extendability-clique-expansion} we will use the following lemma. 

\begin{lemma}[{\cite[Lemma~4.5]{LMR23unif}}]\label{LEMMA:greedily-embedding-Gi}
Fix a real $\eta \in (0, 1)$ and integers $m, n\ge 1$.
Let $\mathcal{G}$ be an $r$-graph with vertex set~$[m]$ and let $\mathcal{H}$ be a further $r$-graph
with $v(\mathcal{H})=n$.
Consider a vertex partition $V(\mathcal{H}) = \bigcup_{i\in[m]}V_i$ and the associated
blowup $\widehat{\mathcal{G}} := \mathcal{G}(V_1,\ldots,V_{m})$ of $\mathcal{G}$.
Suppose that two sets $T \subseteq [m]$ and $S\subseteq V(\mathcal{H})$
have the properties
\begin{enumerate}[label=(\roman*)]
\item\label{it:47a} $|V_{j}'| \ge 2q(|S|+1)|T|\eta^{1/r} n$  for all $j \in T$,
\item\label{it:47b} $|\mathcal{H}[V_{j_1}',\ldots,V_{j_r}']| \ge |\widehat{\mathcal{G}}[V_{j_1}',\ldots,V_{j_r}']|
		- \eta n^r$ for all $\{j_1,\ldots,j_r\} \in \binom{T}{r}$, and
\item\label{it:47c} $|L_{\mathcal{H}}(v)[V_{j_1}',\ldots,V_{j_{r-1}}']| \ge |L_{\widehat{\mathcal{G}}}(v)[V_{j_1}',\ldots,V_{j_{r-1}}']|
		- \eta n^{r-1}$ for all $v\in S$ and for all $\{j_1,\ldots,j_{r-1}\} \in \binom{T}{r-1}$,
\end{enumerate}
where $V_i' := V_i\setminus S$ for $i\in [m]$. 
Then there exists a selection of $q$-set $U_i \subseteq V_j$ for all $j\in [T]$
such that $U := \bigcup_{j\in T}U_j$ satisfies
$\widehat{\mathcal{G}}[U] \subseteq \mathcal{H}[U]$ and
$L_{\widehat{\mathcal{G}}}(v)[U] \subseteq L_{\mathcal{H}}(v)[U]$ for all $v\in S$.
In particular, if $\mathcal{H} \subseteq \widehat{\mathcal{G}}$,
then $\widehat{\mathcal{G}}[U] = \mathcal{H}[U]$ and
$L_{\widehat{\mathcal{G}}}(v)[U] = L_{\mathcal{H}}(v)[U]$ for all $v\in S$.
\end{lemma}

The following fact follows from a simple greedy argument. 

\begin{fact}\label{FACT:extend-F}
    Let $r \ge 3$ be an integer and $F$ be a graph. 
    Suppose that $\mathcal{H}$ is an $r$-graph and $\phi\colon F \to \partial_{r-2}\mathcal{H}$ is an embedding such that for all but at most one edges $e\in F$ the link $L_{\mathcal{H}}(\phi(e))$ contains at least $v(F) + (r-2) |F|$ pairwise disjoint edges. 
    Then $F \subset \mathcal{H}$. 
\end{fact}

\begin{proof}[Proof of Lemma~\ref{LEMMA:vertex-extendability-clique-expansion}]
    First, we prove that $|e\cap V_i| \le 1$ for every $e\in L_{\mathcal{H}}(v_{\ast})$ and every $i\in [\ell]$. 
    Suppose to the contrary that there exists $e_{\ast} \in L_{\mathcal{H}}(v_{\ast})$ and $i_{\ast} \in [\ell]$ such that $|e_{\ast} \cap V_{i_{\ast}}| \ge 2$. By symmetry, we may assume that $i_{\ast} = 1$.  
    Fix distinct vertices $u_1, u_1' \in e_{\ast} \cap V_{i_{\ast}}$. 
    Since 
    \begin{enumerate}[label=(\roman*)]
        \item $|V_i| \ge \varepsilon n \gg \delta n$ for every $i\in [2, \ell]$,  
        \item $|L_{\mathcal{H}'}(u_i)[V_{j_1}, \ldots, V_{j_{r-1}}]| \ge |L_{K^r}(u_j)[V_{j_1}, \ldots, V_{j_{r-1}}]| - \delta n^{r-1}$ for every $i\in \{1,2\}$ and $\{j_1, \ldots, j_{r-1}\} \in \binom{[2,\ell]}{r-1}$, 
        \item $|L_{\mathcal{H}'}(u)[V_{j_1}, \ldots, V_{j_{r-1}}]| \ge |L_{K^r}(u)[V_{j_1}, \ldots, V_{j_{r-1}}]| - \delta n^{r-1}$ for every $\{j_1, \ldots, j_{r-1}\} \in \binom{[2,\ell]}{r-1}$, $i\in [2,\ell] \setminus \{j_1, \ldots, j_{r-1}\}$, and $u\in N_i$,
    \end{enumerate}
    it follows from Lemma~\ref{LEMMA:greedily-embedding-Gi} that there exists a set $U_i \subset V_{i}'$ of size $2C_{F}$ such that the induced subgraph of $\mathcal{H}$ on $U \coloneqq \{u_1, u_1'\} \cup U_2 \cup \cdots \cup U_{\ell}$ is complete $\ell$-partite. 
    Let $H \coloneqq \partial_{r-2}\mathcal{H}[U] \cup \{u_1u_1'\}$. 
    Notice that there exists an embedding $\phi\colon F \to H$ with $\{u_1, u_1'\}$ being a critical edge in $\phi(F)$. 
    Since for every $e\in H\setminus \{u_1u_1'\}$ the link $L_{\mathcal{H}}(e)$ contains at least $C_{F}$ pairwise disjoint edges, it follows from Fact~\ref{FACT:extend-F} that $F \subset \mathcal{H}$, a contradiction.   
    
    Now suppose that $|I_{\mathrm{large}}| > \ell-2$. We aim to show that $|I| = |I_{\mathrm{large}}| = \ell-1$. 
    Suppose to the contrary that $I = [\ell]$ and  $|I_{\mathrm{large}}| \ge \ell-1$. 
    By symmetry, we may assume that $[2,\ell] \subset I_{\mathrm{large}}$. 
    Fix $u_1 \in N_{\mathcal{H}}(v_{\ast}) \cap V_1$ and assume that $E_{\ast}$ is an edge in $\mathcal{H}$ containing $\{v_{\ast}, u_1\}$. 
    Let $N_i \coloneqq N_{\mathcal{H}}^{C_{F}}(v_{\ast}) \cap V_i$ for $i\in [2,\ell]$.
    Recall from the definition of $I_{\mathrm{large}}$ that $|N_i| \ge \hat{\varepsilon} n$ for $i \in [2,\ell]$. 
    Let $\hat{K}^r$ denote the induced subgraph of $K^r$ on $\{u_1\} \cup N_2 \cup \cdots \cup N_{\ell}$. 
    Since 
    \begin{enumerate}[label=(\roman*)]
        \item $|N_i| \ge \hat{\varepsilon} n \gg \delta n$ for every $i\in [1, \ell]$ and  
        \item $|L_{\mathcal{H}'}(u_1)[N_{j_1}, \ldots, N_{j_{r-1}}]| \ge |L_{\hat{K}^r}(u_1)[N_{j_1}, \ldots, N_{j_{r-1}}]| - \delta n^{r-1}$ for every $\{j_1, \ldots, j_{r-1}\} \in \binom{[2,\ell]}{r-1}$, 
        \item $|L_{\mathcal{H}'}(u)[N_{j_1}, \ldots, N_{j_{r-1}}]| \ge |L_{\hat{K}^r}(u)[N_{j_1}, \ldots, N_{j_{r-1}}]| - \delta n^{r-1}$ for every $\{j_1, \ldots, j_{r-1}\} \in \binom{[2,\ell]}{r-1}$, $i\in [2,\ell] \setminus \{j_1, \ldots, j_{r-1}\}$, and $u\in N_i$, 
    \end{enumerate}
    it follows from Lemma~\ref{LEMMA:greedily-embedding-Gi} that there exists a set $U_i \subset N_i$ of size $2C_F$ for every $i\in [2,\ell]$ such that the induced subgraph of $\mathcal{H}$ on $U \coloneqq \{u_1\}\cup U_2 \cup \cdots \cup U_{\ell}$ is complete $\ell$-partite. 
    Notice that the induced subgraph, denoted as $H$, of $\partial_{r-2}\mathcal{H}$ on $U\cup \{v_{\ast}\}$ is complete $(\ell+1)$-partite.  So the edge-critical graph $F$ can be embedded into $H$ (with $\{v_{\ast}, u_1\}$ being a critical edge). 
    In addition, since for every $e\in H\setminus\{u_1 v\}$ the link $L_{\mathcal{H}}(e)$ contains at least $C_{F}$ pairwise disjoint edges, it follows from Fact~\ref{FACT:extend-F} that $F \subset \mathcal{H}$, a contradiction.  
    Therefore, $|I| = |I_{\mathrm{large}}| = \ell-1$. 
\end{proof}

Now we are ready to prove Proposition~\ref{PROP:Q-vtx-ext-F-expansion}. 

\begin{proof}[Proof of Proposition~\ref{PROP:Q-vtx-ext-F-expansion}]
    Let $\varepsilon>0$ be sufficiently small and $n$ be sufficiently large. 
    Let $C_F \coloneqq v(F) + (r-2) |F|$. 
    Let $\mathcal{H}$ be an $n$-vertex $H_{F}^{r}$-free $r$-graph with 
    \begin{align*}
        \delta_{Q}(\mathcal{H}) 
        \ge \left(\lambda_{Q}(K_{\ell}^{r})-\varepsilon \right) \cdot v(Q) \cdot n^{v(Q)-1}, 
    \end{align*}
    and $v_{\ast} \in V(\mathcal{H})$ be a vertex such that $\mathcal{H}_1 \coloneqq  \mathcal{H}-v_{\ast}$ is $\ell$-partite.  
    Let $V \coloneqq V(\mathcal{H}) \setminus \{v_{\ast}\}$ and $V_1 \cup \cdots \cup V_{\ell} = V$ be a partition such that $\mathcal{H}_1$ is a subgraph of $K^r \coloneqq K^r[V_1, \ldots, V_{\ell}]$. 
    Let $x_i \coloneqq |V_i|/n$ for $i\in [\ell]$ and let $x_{\ast} \coloneqq \min\{x_i \colon i\in [\ell]\}$. 
    Note that 
    \begin{align*}
        \delta_{Q}(\mathcal{H}_1) 
        \ge \delta_{Q}(\mathcal{H}) - v(Q) n^{v(Q)-2}
        \ge \left(\lambda_{Q}(K_{\ell}^{r})-2\varepsilon \right) \cdot v(Q) \cdot n^{v(Q)-1}. 
    \end{align*}
    
    Let $P(X_1, \ldots, X_{\ell}) \coloneqq P_{Q,K_{\ell}^{r}}(X_1, \ldots, X_{\ell})$. 
    By Fact~\ref{FACT:Lagrangian-meaning}, 
    \begin{align*}
        \mathrm{inj}(Q, K^r) 
        & = P(x_1, \ldots, x_{\ell}) \cdot n^{v(Q)} + O(n^{v(Q)-1}), \quad\text{and} \\ 
        d_{Q,K^r}(u) 
        & = D_{i}P(x_1, \ldots, x_{\ell}) \cdot n^{v(Q)-1} + O(n^{v(Q)-2}) 
        \quad\text{for every $i\in [\ell]$ and $u\in V_i$}. 
    \end{align*}
    The first equality above implies that 
    \begin{align*}
        P(x_1, \ldots, x_{\ell})
         = \frac{\mathrm{inj}(Q, K^r)}{n^{v(Q)}} - o(1) 
        & \ge \frac{\mathrm{inj}(Q, \mathcal{H}_1)}{n^{v(Q)}} - o(1) \\
        & \ge \frac{1}{v(Q)} \frac{\sum_{v\in V} d_{Q,\mathcal{H}_1}(v)}{n^{v(Q)}} -o(1) \\
        & \ge \lambda_{Q}(K_{\ell}^{r})-2\varepsilon - o(1) 
        \ge \lambda_{Q}(K_{\ell}^{r})-3\varepsilon. 
    \end{align*}
    By Lemma~\ref{LEMMA:Lagrangian-Multiplier}, there exists $(y_1, \ldots, y_{\ell}) \in \mathrm{Opt}(Q,K_{\ell}^{r})$ such that $\max_{i\in [\ell]}|x_i-y_i| \le \delta$ and 
    \begin{align*}
        D_{i}P(x_1, \ldots, x_{\ell}) 
         \le D_{i}P(y_1, \ldots, y_{\ell}) + \delta 
         = v(Q) \cdot \lambda_{Q}(K_{\ell}^{r}) + \delta
         \quad\text{for every}\quad i\in [\ell]. 
    \end{align*}
    This implies that for every $u\in V$, 
    \begin{align*}
        d_{Q,K^r}(u) 
         = D_{i}P(x_1, \ldots, x_{\ell}) \cdot n^{v(Q)-1} + O(n^{v(Q)-2})  
         \le \left(v(Q) \cdot \lambda_{Q}(K_{\ell}^{r}) + 2\delta\right) \cdot n^{v(Q)-1}. 
    \end{align*}
    By Lemma~\ref{LEMMA:homo-subgraph-complete-partite}~\ref{LEMMA:homo-subgraph-complete-partite-2}, for every $u\in V$, 
    \begin{align*}
        d_{K^r}(u) - d_{\mathcal{H}_1}(u)
        & \le \frac{\left(v(Q) \cdot \lambda_{Q}(K_{\ell}^{r}) + 2\delta\right) \cdot n^{v(Q)-1} - \left(\lambda_{Q}(K_{\ell}^{r})-2\varepsilon \right) \cdot v(Q) \cdot n^{v(Q)-1}}{\left(x_{\ast} n/2\right)^{v(Q)-r}} \\
        & \le \frac{2\delta + 2\varepsilon v(Q)}{\left(x_{\ast} /2\right)^{v(Q)-r}} n^{r-1}
        \le \frac{2\delta + 2\varepsilon v(Q)}{\left(\varepsilon_{Q,\ell}/2 - \delta \right)^{v(Q)-1}} n^{r-1}
        \le \delta_1 n^{r-1}. 
    \end{align*}
    In summary, the $r$-graph $\mathcal{H}$ satisfies $\mathcal{H}-v_{\ast} \subset K^{r}(V_1, \ldots, V_{\ell})$ and $d_{\mathcal{H}_1}(u) \ge d_{K^r}(u) - \delta_1 n^{r-1}$ for every $u\in V$. 
    Let 
    \begin{align*}
            I
             \coloneqq 
            \left\{i\in [\ell] \colon N_{\mathcal{H}}(v) \cap V_i \neq \emptyset \right\} \quad\text{and}\quad
            I_{\mathrm{large}}
             \coloneqq 
            \left\{i\in [\ell] \colon |N_{\mathcal{H}}^{C_{F}}(v) \cap V_i| \ge \hat{\varepsilon} n\right\}
    \end{align*}
    It follows from Lemma~\ref{LEMMA:vertex-extendability-clique-expansion} that 
    \begin{enumerate}[label=(\roman*)]
        \item\label{vtx-extend-1} every $e\in L_{\mathcal{H}}(v_{\ast})$ satisfies $|e\cap V_i| \le 1$ for $i\in [\ell]$, and 
        \item\label{vtx-extend-2} either $|I_{\mathrm{large}}| \le \ell-2$ or $|I| = |I_{\mathrm{large}}| = \ell-1$.
    \end{enumerate}
    Suppose that $|I_{\mathrm{large}}| \le \ell-2$. 
    By symmetry, we may assume that $I_{\mathrm{large}} = [\ell']$ for some $\ell' \le \ell-2$. 
    Let 
    \begin{align*}
        \mathcal{E}_{1}
        & \coloneqq \left\{e\cup \{v_{\ast}\} \colon e\in L_{\mathcal{H}}(v_{\ast}) \text{ and } e\cap \left(V_{\ell'+1} \cup \cdots \cup V_{\ell}\right) \neq \emptyset\right\}
        \quad\text{and}\quad \\
        \mathcal{E}_2 
        & \coloneqq \left\{e\cup \{v_{\ast}\} \colon e\in L_{\mathcal{H}}(v_{\ast}) \text{ and } e\cap \left(V_{\ell'+1} \cup \cdots \cup V_{\ell}\right) = \emptyset\right\}. 
    \end{align*}
    Let $\mathcal{H}_2 \coloneqq \mathcal{H} \setminus \mathcal{E}_{1}$, and notice from~\ref{vtx-extend-1} that $\mathcal{H}_2$ is $\ell$-partite. Fix $v \in V_{\ell}$ and observe that $L_{\mathcal{H}_2}(v_{\ast}) \subset L_{K^r}(v)$. 
    It follows from Lemma~\ref{LEMMA:homo-subgraph-complete-partite} that 
    \begin{align}\label{equ:vtx-extend-1}
        d_{Q,\mathcal{H}_2}(v_{\ast})
        & \le d_{Q,K^r}(v_{\ast}) - \left(d_{K^r}(v) - d_{\mathcal{H}_2}(v_{\ast})\right) \left(\frac{x_{\ast} n}{2}\right)^{v(Q)-r} \notag \\
        & \le \left(v(Q) \cdot \lambda_{Q}(K_{\ell}^{r}) + 2\delta\right) \cdot n^{v(Q)-1} - \left(x_{\ast}n\right)^{r-1} \left(\frac{x_{\ast} n}{2}\right)^{v(Q)-r}. 
    \end{align}
    
    It follows from the definition of $I_{\mathrm{large}}$ that 
    \begin{align*}
        |\mathcal{E}_1| 
        \le |V_{\ell'+1} \cup \cdots \cup V_{\ell}| \cdot C_{F}\binom{n}{r-3} + (\ell-\ell') \cdot \hat{\varepsilon} n \cdot \binom{n}{r-2} 
        \le 2\hat{\varepsilon} \ell n^{r-1}. 
    \end{align*}
    Therefore, the set $\left\{\phi \in \mathrm{Inj}(Q,\mathcal{H}) \colon \exists e\in \mathcal{E}_1 \text{ such that } e \in \phi(Q)\right\}$
    has size at most $|\mathcal{E}_1| n^{v(Q)-r} \le 2\hat{\varepsilon} \ell n^{v(Q)-1}$. 
    Combined with~\eqref{equ:vtx-extend-1}, we obtain 
    \begin{align*}
        d_{Q,\mathcal{H}}(v)
        & \le \left(v(Q) \cdot \lambda_{Q}(K_{\ell}^{r}) + 2\delta\right) \cdot n^{v(Q)-1} - \left(x_{\ast}n\right)^{r-1} \left(\frac{x_{\ast} n}{2}\right)^{v(Q)-r} 
        + 2\hat{\varepsilon} \ell n^{v(Q)-1} \\
        & < \delta_{Q}(\mathcal{H}), 
    \end{align*}
    a contradiction. 
    Therefore, $|I| = |I_{\mathrm{large}}| = \ell-1$, and hence, $\mathcal{H}$ is $\ell$-partite. 
\end{proof}

\section{Proof of Theorem~\ref{THM:Homomorphism-complete-multipartite-hypergraph}}\label{SEC:proof-complete-multipartite}
%
Fix integers $\ell \ge k \ge r\ge 2$, fix an $r$-graph $Q$ that is a blowup of some $k$-vertex $2$-covered $r$-graph, and fix an edge-critical graph $F$ with $\chi(F) = \ell+1$.
By Fact~\ref{FACT:weak-expansion},  $\mathcal{K}_{\ell+1}^{r} \le_{\mathrm{hom}} H_{F}^{r}$, $\mathcal{K}_{\ell+1}^{r}$ is blowup-invariant, and $\mathcal{K}_{\ell+1}^{r}$ is symmetrized-stable with respect to $\mathfrak{K}_{\ell}^{r}$. 

First, we show that $Q$ is symmetrization-increasing in the following result. 

\begin{proposition}\label{PROP:complete-multipartite-sym-increase}
    The map $\Gamma \colon \mathfrak{G}^{r} \to \mathbb{R}$ defined by $\Gamma(\mathcal{H}) \coloneqq \mathrm{inj}(Q,\mathcal{H})$ is symmetrization-increasing.
\end{proposition}
\begin{proof}[Proof of Proposition~\ref{PROP:complete-multipartite-sym-increase}]
    Let $\mathcal{H}_0 \coloneqq \mathcal{H}$, $\mathcal{H}_1 \coloneqq \mathcal{H}_{u\to v}$, and $\mathcal{H}_2 \coloneqq \mathcal{H}_{v\to u}$.
    For $i\in \{0,1,2\}$, let 
    \begin{align*}
        \Phi^{i}_{\overline{u}, \overline{v}} 
        & \coloneqq \left\{\phi\in \mathrm{Inj}(Q,\mathcal{H}_i) \colon \{u,v\} \cap \phi(V(Q)) = \emptyset\right\}, \\
        \Phi^{i}_{u,\overline{v}} 
        & \coloneqq \left\{\phi\in \mathrm{Inj}(Q,\mathcal{H}_i) \colon u \in \phi(V(Q)),\ v\not\in \phi(V(Q))\right\}, \\
        \Phi^{i}_{\overline{u},v} 
        & \coloneqq \left\{\phi\in \mathrm{Inj}(Q,\mathcal{H}_i) \colon u \not\in \phi(V(Q)),\ v\in \phi(V(Q))\right\}, \\
        \Phi^{i}_{u,v} 
        & \coloneqq \left\{\phi\in \mathrm{Inj}(Q,\mathcal{H}_i) \colon \{u,v\} \subset \phi(V(Q))\right\}. 
    \end{align*}
    It is easy to see from the definition that 
    \begin{align*}
         \Phi^{0}_{\overline{u}, \overline{v}} 
        = \Phi^{1}_{\overline{u}, \overline{v}} 
        = \Phi^{2}_{\overline{u}, \overline{v}}, \quad
         2|\Phi^{0}_{u,\overline{v}}| 
        = |\Phi^{2}_{u,\overline{v}}| + |\Phi^{2}_{\overline{u},v}|, \quad\text{and}\quad 
        2|\Phi^{0}_{\overline{u},v}| 
        = |\Phi^{1}_{u,\overline{v}}| + |\Phi^{1}_{\overline{u},v}|. 
    \end{align*}

    \begin{claim}
        $|\Phi^{0}_{u,v}| \le \min\left\{ |\Phi^{1}_{u,v}|,\  |\Phi^{2}_{u,v}| \right\}$. 
    \end{claim}
    \begin{proof}
        Let us assume that $F$ is the $2$-covered $r$-graph on $[k]$ such that $Q$ is a blowup of $F$. 
        Let $W_1 \cup \cdots \cup W_{k} = V(Q)$ be a partition such that $Q = F[W_1, \ldots, W_{k}]$. 
        Fix $\phi \in \Phi^{0}_{u,v}$. 
        Since $\{u,v\}$ is uncovered in $\mathcal{H}$, we know that $\phi^{-1}(\{u,v\}) \subset W_i$ for some $i\in [k]$. 
        Since $F$ is $2$-covered, $\phi^{-1}(u)$ and $\phi^{-1}(v)$ are twins in $Q$, that is, $L_{Q}(\phi^{-1}(u)) = L_{Q}(\phi^{-1}(v))$. 
        Therefore, every member $e\in L_{Q}(\phi^{-1}(u))$ satisfies $\phi(e) \in L_{\mathcal{H}}(u) \cap L_{\mathcal{H}}(v)$. 
        It follows that $\phi \in \Phi^{1}_{u,v}$ and $\phi \in \Phi^{2}_{u,v}$. 
        Therefore, $|\Phi^{0}_{u,v}| \le \min\left\{ |\Phi^{1}_{u,v}|,\  |\Phi^{2}_{u,v}| \right\}$. 
    \end{proof}
        
    It follows from the argument above that $2\Gamma(\mathcal{H}) \le \Gamma(\mathcal{H}_1) + \Gamma(\mathcal{H}_2)$, and hence, $\Gamma$ is symmetrization-increasing. 
\end{proof}

Proposition~\ref{PROP:Q-vtx-ext-F-expansion} established that $H_{F}^{r}$ is $Q$-vertex-extendable with respect to $\mathfrak{K}_{\ell}^{r}$. 
By Theorem~\ref{THM:general-generalized-Turan-a}, to establish Theorem~\ref{THM:Homomorphism-complete-multipartite-hypergraph}, it suffices to show that both $\mathcal{K}_{\ell}^{r}$ and $H_{F}^{r}$ are $Q$-vertex-extendable with respect to $\mathfrak{K}_{\ell}^{r}$. 

\begin{proposition}\label{PROP:complete-multipartite-vtx-extend}
    Both $\mathcal{K}_{\ell}^{r}$ and $H_{F}^{r}$ are $Q$-vertex-extendable with respect to $\mathfrak{K}_{\ell}^{r}$.
\end{proposition}
\begin{proof}[Proof of Proposition~\ref{PROP:complete-multipartite-vtx-extend}]
    Due to Proposition~\ref{PROP:Q-vtx-ext-F-expansion}, it suffices to show that $\pi_{\mathrm{inj}}(Q, H_{F}^{r}) =  \pi_{\mathrm{inj}}(Q, \mathcal{K}_{\ell}^{r}) = \lambda_{Q}(\mathcal{K}_{\ell}^{r})$. 
    Let $\Gamma \colon \mathfrak{G}^{r} \to \mathbb{R}$ be defined as 
    \begin{align*}
        \Gamma(\mathcal{H})
        \coloneqq \mathrm{inj}(Q,\mathcal{H}) \cdot \mathbbm{1}_{\mathcal{K}_{\ell}^{r}}(\mathcal{H}). 
    \end{align*}
    Since $\mathcal{K}_{\ell}^{r}$ is blowup-invariant, it follows from Proposition~\ref{PROP:complete-multipartite-sym-increase} and Fact~\ref{FACT:blowup-invariant-symmetrization-increasing} that $\Gamma$ is symmetrization-increaing. 
    By Fact~\ref{FACT:weak-expansion}, $\Gamma$ is symmetrizaed-stable with respect to $\mathfrak{K}_{\ell}^{r}$. 
    Therefore, by Fact~\ref{FACT:exact-symm-increase}, 
    \begin{align*}
        \mathrm{inj}(n,Q,\mathcal{K}_{\ell}^{r}) = \mathrm{ex}_{\Gamma}(n) = \max\left\{\Gamma(\mathcal{H}) \colon \mathcal{H} \in \mathfrak{K}_{\ell}^{r} \text{ and } v(\mathcal{H}) = n\right\}
    \end{align*}
    In particular, if follows from Fact~\ref{FACT:Lagrangian-meaning} that $\pi_{\mathrm{inj}}(Q, \mathcal{K}_{\ell}^{r}) = \lambda_{Q}(\mathcal{K}_{\ell}^{r})$. 
    Since $\mathcal{K}_{\ell}^{r} \le_{\mathrm{hom}} H_{F}^{r}$, it follows from Fact~\ref{FACT:le-hom} that $\pi_{\mathrm{inj}}(Q, H_{F}^{r}) \le \pi_{\mathrm{inj}}(Q, \mathcal{K}_{\ell}^{r}) = \lambda_{Q}(\mathcal{K}_{\ell}^{r})$. 
    On the other hand, since $\mathfrak{K}_{\ell}^{r}$ is also $H_{F}^{r}$-free, we obtain $\pi_{\mathrm{inj}}(Q, H_{F}^{r}) \ge \lambda_{Q}(\mathcal{K}_{\ell}^{r})$,  and hence, $\pi_{\mathrm{inj}}(Q, H_{F}^{r}) = \lambda_{Q}(\mathcal{K}_{\ell}^{r})$. 
\end{proof}
\section{Proof of Theorem~\ref{THM:Turan-good-edge-critical-expansion}}\label{SEC:Proof-Turan-goodness}
%
Let $\ell \ge r \ge 2$ be integers. 
Let $Q$ be an $r$-graph and $F$ be a graph with $\chi(F) = \ell+1$. 
Recall that $H_{F}^{r}$ is $Q$-edge-stable with respect to $\mathfrak{K}_{\ell}^{r}$ if every $\mathcal{F}$-free $n$-vertex $r$-graph $\mathcal{H}$ with $\mathrm{inj}(Q,\mathcal{H}) = (1-o(1))\cdot \mathrm{inj}(n,Q,\mathcal{F})$ is $\ell$-partite after removing $o(n^r)$ edges.
Following the definition of Gerbner--Palmer~\cite{GP22}, we refine the definition of edge-stable by saying that $Q$ is \textbf{$H_{F}^{r}$-Tur\'{a}n-stable} if for every $\varepsilon>0$ there exist $\delta>0$ and $N_0$ such that every $H_{F}^{r}$-free $r$-graph $\mathcal{H}$ on $n \ge N_0$ vertices with $\mathrm{inj}(Q, \mathcal{H}) \ge (1-\delta)\cdot \mathrm{inj}(n,Q)$ is a copy of $T^{r}(n,\ell)$ after deleting and adding at most $\varepsilon n^r$ edges. 

First, we establish the following reduction theorem using Fact~\ref{FACT:homomorphism-shadow}. 

\begin{theorem}\label{THM:Turan-stable-expansion}
    Let $\ell \ge r \ge 2$ be integers, $Q$ be an $r$-graph, and $F$ be a graph with $\chi(F) = \ell+1$.
    If $K_{\ell+1}$ is $\partial_{r-2} Q$-edge-stable with respect to $\mathfrak{K}_{\ell}^{2}$, then $H_{F}^{r}$ is $Q$-edge-stable with respect to $\mathfrak{K}_{\ell}^{r}$. 
    Moreover, if $\partial_{r-2}Q$ is $K_{\ell+1}$-Tur\'{a}n-stable, then $Q$ is $H_{F}^{r}$-Tur\'{a}n-stable.    
\end{theorem}

In~\cite{GK23}, Gerbner--Karim proved that $K_{\ell+1}$ is $Q$-edge-stable with respect to $\mathfrak{K}_{\ell}^{r}$ if $Q$ is complete $k$-partite for some $k \le \ell$ (see also Theorem~\ref{THM:Homomorphism-complete-multipartite-hypergraph} for the case $r= 2$ and $F = K_{\ell+1}$). 
Additionally, they proved that $Q$ is $K_{\ell+1}$-Tur\'{a}n-stable for every $\ell \ge 300 \left(v(Q)\right)^{9}$. 
Based on these stability theorems, we obtain the following corollary of Theorem~\ref{THM:Turan-stable-expansion}. 

\begin{corollary}\label{CORO:Turan-stable-expansion}
    Let $\ell \ge r \ge 2$ be integers, $Q$ be an $r$-graph, and $F$ be a graph with $\chi(F) = \ell+1$. 
    \begin{enumerate}[label=(\roman*)]
        \item\label{CORO:Turan-stable-expansion-1} If $\partial_{r-2} Q$ is a complete $k$-partite graph for some $k \le \ell$, then $H_{F}^{r}$ is $Q$-edge-stable with respect to $\mathfrak{K}_{\ell}^{r}$. 
        \item\label{CORO:Turan-stable-expansion-2} If $\ell \ge 300 \left(v(Q)\right)^{9}$, then $Q$ is $H_{F}^{r}$-Tur\'{a}n-stable. 
    \end{enumerate}
\end{corollary}

In the proof of Theorem~\ref{THM:Turan-stable-expansion}, we will use the following simple lemma. 
\begin{lemma}\label{LEMMA:weak-stable-rough-upper-bound}
    Let $Q$ be an $r$-graph and $F$ be a graph with $\chi(F) = \ell + 1 > r$. 
    Suppose that $F$ is $\partial_{r-2} Q$-edge-stable with respect to $\mathfrak{K}_{\ell}^{2}$. Then 
    \begin{align*}
        \left|\mathrm{inj}(n,Q,H_{F}^{r}) - \mathrm{inj}(n,\partial_{r-2} Q,F) \right| 
        = o(n^{v(Q)}). 
    \end{align*}
    In particular, $\pi_{\mathrm{inj}}(Q, H_{F}^{r}) = \lambda_{\partial_{r-2}Q}(K_{\ell}^{2}) =  \lambda_{Q}(K_{\ell}^{r})$. 
\end{lemma}

The following simple fact will be useful. 

\begin{fact}[see e.g.~{\cite[Lemma~3]{PI13}}]\label{FACT:expansion-reduction}
    Let $r \ge 3$ be an integer and $F$ be a graph. 
    Every $n$-vertex $H_{F}^{r}$-free $r$-graph $\mathcal{H}$ contains a subgraph $\mathcal{H}'$ such that $\partial_{r-2} \mathcal{H}'$ is $F$-free and 
    \begin{align*}
        |\mathcal{H}'| 
        \ge |\mathcal{H}| - \left(v(F)+(r-2)|F|\right)\binom{n}{r-3}\binom{n}{2}
        \ge |\mathcal{H}| - r\left(v(F)\right)^2 n^{r-1}. 
    \end{align*}
\end{fact}

\begin{proof}[Proof of Lemma~\ref{LEMMA:weak-stable-rough-upper-bound}]
    First we prove that $\mathrm{inj}(n, Q, H_{F}^{r}) \ge \mathrm{inj}(n,\partial_{r-2}Q, F) - o(n^{v(Q)})$.
    Let $n$ be a sufficiently large integer. 
    Since $F$ is $\partial_{r-2} Q$-edge-stable with respect to $\mathfrak{K}_{\ell}^{2}$, there exists a complete $\ell$-partite graph $G$ on $n$ vertices such that 
    \begin{align}\label{equ:r-to-2-a}
        \mathrm{inj}(n,\partial_{r-2}Q, F)
        \le \mathrm{inj}(\partial_{r-2}Q, G) + o(n^{v(Q)}). 
    \end{align}
    Let $\mathcal{H}_{G} \coloneqq \left\{S\in \binom{V(G)}{r} \colon G[S] \cong K_{r}\right\}$.  
    It is clear that the $r$-graph $\mathcal{H}_{G}$ is complete $\ell$-partite and $H_{F}^{r}$-free (since $G$ is complete $\ell$-partite and $F$-free). 
    So it follows from Fact~\ref{FACT:homomorphism-shadow} and~\eqref{equ:r-to-2-a} that  
    \begin{align*}
        \mathrm{inj}(n, Q, H_{F}^{r})
        \ge \mathrm{inj}(Q, \mathcal{H}_{G})
        = \mathrm{inj}(\partial_{r-2}Q, G)
        \ge \mathrm{inj}(n,\partial_{r-2}Q, F) - o(n^{v(Q)}). 
    \end{align*}

    Next, we prove that $\mathrm{inj}(n, Q, H_{F}^{r}) \le \mathrm{inj}(n,\partial_{r-2}Q, F) + o(n^{v(Q)})$. 
    Let $\mathcal{H}$ be an $n$-vertex $H_{F}^{r}$-free $r$-graph with $\mathrm{inj}(Q,\mathcal{H}) = \mathrm{inj}(n,Q,H_{F}^{r})$. 
    It follows from Fact~\ref{FACT:expansion-reduction} that there exists $\mathcal{H}' \subset \mathcal{H}$ such that $|\mathcal{H}'| \ge |\mathcal{H}| - r\left(v(F)\right)^2 n^{r-1} = |\mathcal{H}| - o(n^r)$ and $\partial_{r-2}\mathcal{H}'$ is $F$-free. 
    In particular, 
    \begin{align*}
        \mathrm{inj}(Q, \mathcal{H}')
        \ge \mathrm{inj}(Q, \mathcal{H}) - |\mathcal{H}\setminus \mathcal{H}'| n^{v(Q)-r}
        = \mathrm{inj}(n,Q,H_{F}^{r}) - o(n^{v(Q)}). 
    \end{align*}
    Combined with Fact~\ref{FACT:homomorphism-shadow}, we obtain 
    \begin{align*}
        \mathrm{inj}(n,Q,H_{F}^{r}) - o(n^{v(Q)})
        \le \mathrm{inj}(Q, \mathcal{H}')
        \le \mathrm{inj}(\partial_{r-2}Q, \partial_{r-2}\mathcal{H}')
        \le \mathrm{inj}(n,\partial_{r-2}Q, F). 
    \end{align*}
    This proves $\mathrm{inj}(n, Q, H_{F}^{r}) \le \mathrm{inj}(n,\partial_{r-2}Q, F) + o(n^{v(Q)})$, and hence, completing the proof of Lemma~\ref{LEMMA:weak-stable-rough-upper-bound}. 
\end{proof}
Now we are ready to prove Theorem~\ref{THM:Turan-stable-expansion}. 
\begin{proof}[Proof of Theorem~\ref{THM:Turan-stable-expansion}]
    Here we present only the proof for the second part of Theorem~\ref{THM:Turan-stable-expansion}, since the proof for the first part is quite similar and requires only minor modifications.
    Fix $\varepsilon >0$. 
    Let $\delta > 0$ be sufficiently small and $n$ be sufficiently large. 
    Let $\mathcal{H}$ be an $n$-vertex $H_{F}^{r}$-free $r$-graph with $\mathrm{inj}(Q, \mathcal{H}) \ge (1-\delta)\cdot \mathrm{inj}(n,Q,H_{F}^{r})$. 
    For convenience, let $G \coloneqq \partial_{r-2}\mathcal{H}$.

    Applying Fact~\ref{FACT:expansion-reduction} to $\mathcal{H}$, we find a subgraph $\mathcal{H}_1 \subset \mathcal{H}$ with $|\mathcal{H}_1| \ge |\mathcal{H}| - \delta n^r$ such that the graph $G_1 \coloneqq \partial_{r-2}\mathcal{H}_1$ is $F$-free. 
    It follows from Fact~\ref{FACT:homomorphism-shadow} and Lemma~\ref{LEMMA:weak-stable-rough-upper-bound} that 
    \begin{align*}
        \mathrm{inj}(\partial_{r-2}Q, G_1)
        \ge 
        \mathrm{inj}(Q, \mathcal{H}_1)
        \ge \mathrm{inj}(Q, \mathcal{H}) - \delta n^r \cdot n^{v(Q)-r} 
        & \ge \mathrm{inj}(n,Q,H_{F}^{r}) - 2\delta n^{v(Q)} \\
        & \ge \mathrm{inj}(n,\partial_{r-2}Q,F) - 3\delta n^{v(Q)}. 
    \end{align*}
    Since $\chi(\partial_{r-2}Q) = \chi(Q) = \ell+1$, by the Graph Removal Lemma (see e.g.~\cite{RS78}), there exists a $K_{\ell+1}$-free subgraph $G_2 \subset G_1$ such that $|G_2| \ge |G_1| - \delta n^2$. In particular,  
    \begin{align*}
        \mathrm{inj}(\partial_{r-2}Q, G_2)
        \ge \mathrm{inj}(\partial_{r-2}Q, G_1) - \delta n^r \cdot n^{v(Q)-r} 
        \ge \mathrm{inj}(n,\partial_{r-2}Q,F) - 4\delta n^{v(Q)}. 
    \end{align*}
    By assumption, $K_{\ell+1}$ is $\partial_{r-2}Q$-edge-stable with respect to $\mathfrak{K}_{\ell}^{2}$, so there exists a balanced partition $V_1 \cup \cdots \cup V_{\ell} = V$ such that the complete $\ell$-partite graph $K^2 \coloneqq K^2[V_1, \ldots, V_{\ell}]$ satisfies $|G_2 \triangle K^2| \le \varepsilon n^2$. 
    Let $x_{\ast} \coloneqq \min\left\{|V_i|/n \colon i\in [\ell]\right\} = 1/\ell - o(1)$. 
    Let $K^r \coloneqq K^r[V_1, \ldots, V_{\ell}]$ and $\mathcal{H}_2 \coloneqq \mathcal{H}_1 \cap K^r$. 
    Then 
    \begin{align*}
        |\mathcal{H}_2|
        \ge |\mathcal{H}_1| - |G_1 \setminus K^2| \cdot n^{r-2}
        & \ge |\mathcal{H}_1| - \left(|G_1\setminus G_2| + |G_2\setminus K^2|\right) n^{r-2} \\
        & \ge |\mathcal{H}_1| - (\delta+\varepsilon) n^r 
        \ge |\mathcal{H}| -\delta n^r - (\delta+\varepsilon) n^r 
        \ge |\mathcal{H}| - 2\varepsilon n^r. 
    \end{align*}
    Consequently, 
    \begin{align*}
        \mathrm{inj}(Q, \mathcal{H}_{2})
        \ge \mathrm{inj}(Q, \mathcal{H})
            - 2\varepsilon n^r \cdot n^{v(Q)-r}
        \ge \mathrm{inj}(n,Q,H_{F}^{r}) - 3\varepsilon n^{v(Q)}. 
    \end{align*}
    Combined with Lemma~\ref{LEMMA:homo-subgraph-complete-partite}~\ref{LEMMA:homo-subgraph-complete-partite-1} we obtain  
    \begin{align*}
        |K^r \setminus \mathcal{H}_2|
        \le \frac{\mathrm{inj}(Q, K^r) - \mathrm{inj}(Q, \mathcal{H}_2)}{\left(x_{\ast} n/2\right)^{v(Q)-r}}
        \le \frac{3\varepsilon n^{v(Q)}}{\left( n/(4\ell)\right)^{v(Q)-r}}
        \le 3 (4\ell)^{v(Q)-r} \varepsilon n^{r}. 
    \end{align*}
    Therefore, 
    \begin{align*}
        |\mathcal{H} \triangle K^r|
        \le |K^r \setminus \mathcal{H}_2| + |\mathcal{H} \setminus \mathcal{H}_2|
        \le 3 (4\ell)^{v(Q)-r} \varepsilon n^{r} + 2\varepsilon n^{r}
        \le (4\ell)^{v(Q)} \varepsilon n^{r}. 
    \end{align*}
    proving the second part of Theorem~\ref{THM:Turan-stable-expansion}. 
\end{proof}
Now we are ready to present the proofs for Theorems~\ref{THM:Homomorphism-complete-multipartite-hypergraph-b} and~\ref{THM:Turan-good-edge-critical-expansion}. 
\begin{proof}[Proof of Theorem~\ref{THM:Homomorphism-complete-multipartite-hypergraph-b}]
    Fix an integer $r\ge 2$, an $r$-graph $Q$ with $\partial_{r-2}Q$ being complete $k$-partite, and 
    an edge-critical graph $F$ with $\chi(F) = \ell + 1 > k$. 
    By Corollary~\ref{CORO:Turan-stable-expansion}~\ref{CORO:Turan-stable-expansion-1}, $H_{F}^{r}$ is edge-stable with respect to $\mathfrak{K}_{\ell}^{r}$. 
    In particular, $\pi_{\mathrm{inj}}(Q, H_{F}^{r}) = \lambda_{Q}(K_{\ell}^{r})$, which combined with Proposition~\ref{PROP:complete-multipartite-vtx-extend}, implies that $H_{F}^{r}$ is vertex-stable with respect to $\mathfrak{K}_{\ell}^{r}$. 
    Therefore, it follows from Theorem~\ref{THM:general-generalized-Turan-b} that $H_{F}^{r}$ is degree-stable with respect to $\mathfrak{K}_{\ell}^{r}$. 
\end{proof}
\begin{proof}[Proof of Theorem~\ref{THM:Turan-good-edge-critical-expansion}]
    Fix an integer $r\ge 2$, an $r$-graph $Q$ without isolated vertices, and 
    an edge-critical graph $F$ with $\chi(F) = \ell + 1 \ge 300(v(Q))^{9}$. 
    By Corollary~\ref{CORO:Turan-stable-expansion}~\ref{CORO:Turan-stable-expansion-2}, $H_{F}^{r}$ is edge-stable with respect to $\mathfrak{K}_{\ell}^{r}$. 
    In particular, $\pi_{\mathrm{inj}}(Q, H_{F}^{r}) = \lambda_{Q}(K_{\ell}^{r})$, which combined with Proposition~\ref{PROP:complete-multipartite-vtx-extend}, implies that $H_{F}^{r}$ is vertex-stable with respect to $\mathfrak{K}_{\ell}^{r}$. 
    Therefore, it follows from Theorem~\ref{THM:general-generalized-Turan-b} that $H_{F}^{r}$ is degree-stable with respect to $\mathfrak{K}_{\ell}^{r}$. 
\end{proof}
%
\section{Concluding remarks}\label{SEC:Remark}
Given an $r$-graph $Q$ and a family $\mathcal{F}$ of $r$-graphs, let 
\begin{align*}
    \mathrm{hom}(n,Q,\mathcal{F})
    \coloneqq \max\left\{\mathrm{hom}(Q, \mathcal{H}) \colon \text{$\mathcal{H} \in \mathfrak{G}^{r}$ is $\mathcal{F}$-free}\right\}. 
\end{align*}
All results in this paper concerning $\mathrm{inj}(n,Q,\mathcal{F})$ also apply to $\mathrm{hom}(n,Q,\mathcal{F})$ due to Fact~\ref{FACT:inj-Hom}. The proofs require only minor adaptations of the current proofs.

Given integers $r > t \ge 1$ and a real number $p> 1$, we define the \textbf{$(t,p)$-norm} of $\mathcal{H}$ as
\begin{align*}
    \norm{\mathcal{H}}_{t,p}
    \coloneqq \sum_{T\in \binom{V(\mathcal{H})}{t}} d_{\mathcal{H}}^{p}(T), 
\end{align*}
where $d_{\mathcal{H}}^{p}(T)$ is shorthand for $\left(d_{\mathcal{H}}(T)\right)^{p}$. 
For a family $\mathcal{F}$ of $r$-graphs the \textbf{$(t,p)$-Tur\'{a}n number} of $\mathcal{F}$ is 
\begin{align*}
    \mathrm{ex}_{t,p}(n,\mathcal{F})
    \coloneqq \max\left\{\norm{\mathcal{H}}_{t,p} \colon \text{$v(\mathcal{H}) = n$ and $\mathcal{H}$ is $\mathcal{F}$-free}\right\}. 
\end{align*}
Balogh--Clemen--Lidick\'{y}~\cite{BCL22,BCL22b} initiated the study of $\mathrm{ex}_{r-1,2}(n,\mathcal{F})$ for hypergraph families, with most results obtained through computer-assisted Flag Algebra calculations.
In an upcoming paper, we will demonstrate the applications of Theorems~\ref{THM:GeneralFunction-Stability},~\ref{THM:GeneralFunction-Stability-b}, and~\ref{THM:GeneralFunction-Stability-c} in $(t,p)$-norm Tur\'{a}n problems. 
We will establish the degree-stability of the $(t,p)$-norm Tur\'{a}n problem for expansions of edge-critical graphs, applicable to all cases where $r > t \ge 1$ and $p > 1$. 
Additionally, for every $p > 1$, we will establish the degree-stability of the $(2,p)$-norm Tur\'{a}n problem for $F_5 \coloneqq \{\{1,2,3\},\{1,2,4\},\{3,4,5\}\}$.

It is worth mentioning that in the definition of locally $C$-Lipschitz,~\eqref{equ:DEF:general-property-a2} can be further relaxed to$\colon$ 
\begin{align*}
    d_{\Gamma, \mathcal{H}'}(v)
            \ge 
            d_{\Gamma, \mathcal{H}}(v) - 
            f\left(\frac{n-n'}{n}\right) \cdot \mathrm{exdeg}_{\Gamma}(n)
            - o(\mathrm{degex}_{\Gamma}(n))
\end{align*}
for any function $f \colon \mathbb{R} \to \mathbb{R}$ that goes to $0$ as $\frac{n-n'}{n}$ goes to $0$. 
This relaxation will be useful for the $(t,p)$-norm Tur\'{a}n problem when $p \in (1,2)$, where we will choose $f\left(\frac{n-n'}{n}\right) = C  \left(\frac{n-n'}{n}\right)^{\varepsilon}$ for some constants $C, \varepsilon > 0$. 
The current definition (Definition~\ref{DEF:general-property-a}~\ref{DEF:general-property-a2}) is adopted for simplicity in calculations.  

Applications of the degree-stability of Tur\'{a}n problems within Complexity Theory were systematically explored in a recent work~\cite{HLZ24}. 
Exploring the potential applications of the theorems in this paper within Complexity Theory seems like an interesting direction.
\section*{Acknowledgments}
The results in Section~\ref{SUBSEC:intro-general} are partially motivated by a related project proposed by Oleg Pikhurko. We are very grateful to him for his invaluable discussions and comments.
\bibliographystyle{alpha}
\bibliography{LpNorm}
\end{document}